\documentclass[a4paper,10pt]{amsart}

\usepackage[utf8]{inputenc}
\usepackage{graphicx}
\usepackage{amsmath,amsfonts,amsthm,mathtools}
\usepackage{verbatim}

\usepackage{todonotes}

\usepackage{pgfplots}
\pgfplotsset{compat=1.10} 

\usepackage{array} %for bold column in table
\usepackage{hyperref}

\newtheorem{theorem}{Theorem}[section]
\newtheorem{lemma}[theorem]{Lemma}
\newtheorem{corollary}[theorem]{Corollary}

\theoremstyle{definition}
\newtheorem{definition}[theorem]{Definition}

\newtheorem{hypothesisspace}[theorem]{Assumptions on the space}
\newtheorem{hypothesisoperator}[theorem]{Assumptions on the operator}
\newtheorem{hypothesisapproximatingspace}[theorem]{Assumptions on the approximating space}
\newtheorem{normsulam}[theorem]{Norms for the Ulam discretization}
\newtheorem{normspiecewise}[theorem]{Norms for the piecewise linear discretization}

\newtheorem{notation}[theorem]{Notation}
\newtheorem{example}[theorem]{Example}

\theoremstyle{remark}
\newtheorem{remark}[theorem]{Remark}

\DeclarePairedDelimiter{\abs}{\lvert}{\rvert}

\DeclarePairedDelimiter{\norm}{\lVert}{\rVert}

\makeatletter
\newcommand{\trinorm}{\@ifstar\@trinorms\@trinorm}
\newcommand{\@trinorms}[1]{%
  \left|\mkern-1.5mu\left|\mkern-1.5mu\left|
   #1
  \right|\mkern-1.5mu\right|\mkern-1.5mu\right|
}
\newcommand{\@trinorm}[2][]{%
  \mathopen{#1|\mkern-1.5mu#1|}
  #2
  \mathclose{#1|\mkern-1.5mu#1|_{L^1}}
}
\makeatother

\DeclareMathOperator{\Var}{Var}
\DeclareMathOperator{\Lip}{Lip}

\usepackage{algorithm,algorithmicx,algpseudocode}

\title[Rigorous Computation of Invariant Densities]{A general framework for the rigorous computation of invariant densities and the coarse-fine strategy}
\author{S. Galatolo}
\address{Universit\`a di Pisa - Dipartimento di Matematica - Largo Pontecorvo, 2 - 56127, Pisa, Italy}
\email{stefano.galatolo@unipi.it}

\author{M. Monge}
\address{Universidade Federal do Rio de Janeiro - Instituto de Matemática - Av. Athos da Silveira Ramos, 149 - Edifício do Centro de Tecnologia, Bloco C (Térreo) - Cidade Universitária}
\email{maurizio.monge@im.ufrj.br}

\author{I. Nisoli}
\address{Universidade Federal do Rio de Janeiro - Instituto de Matemática - Av. Athos da Silveira Ramos, 149 - Edifício do Centro de Tecnologia, Bloco C (Térreo) - Cidade Universitária}
\address{Department of Mathematics, Hokkaido University, N10 W8, Kita-ku, Sapporo 001-0010, Japan}
\address{RIES, Hokkaido University, N20 W10, Kita-ku, Sapporo 001-0020, Japan}
\email{nisoli@im.ufrj.br, isaia.nisoli@es.hokudai.ac.jp}

\author{F. Poloni}
\address{Universit\`a di Pisa - Dipartimento di Informatica - Largo Pontecorvo, 3 - 56127, Pisa, Italy}
\email{federico.poloni@unipi.it}

\begin{document}

\begin{abstract}
In this paper we present a general, axiomatical framework for the rigorous approximation of invariant densities and other important statistical features of dynamics. We approximate the system trough a finite element reduction, by composing the associated transfer operator with a suitable finite dimensional projection (a discretization scheme) as in the well-known Ulam method.

We introduce a general framework based on a list of properties (of the system and of the projection) that need to be verified so that we can take advantage of a so-called ``coarse-fine'' strategy. This strategy is a novel method in which we exploit information coming from a coarser approximation of the system to get useful information on a finer approximation, speeding up the computation. 
This coarse-fine strategy allows a precise estimation of invariant densities and also allows to estimate rigorously the speed of mixing of the system by the speed of mixing of a coarse approximation of it, which can easily be estimated by the computer.

The estimates obtained here are \emph{rigourous}, i.e., they come with exact error bounds that are guaranteed to hold and take into account both the discretiazation and the approximations induced by finite-precision arithmetic.

We apply this framework to several discretization schemes and examples of invariant density computation from previous works, obtaining a remarkable reduction in computation time.

We have implemented the numerical methods described here in the Julia programming language, and released our implementation publicly as a Julia package.
\end{abstract}

\maketitle

\section{Introduction}

Several important features of the statistical behavior of a dynamical system
are related to the properties of its invariant measures and in particular to
the properties of the so called \emph{\ Physical Invariant Measure}\footnote{%
This is a class of invariant measures representing the statistical behavior
of large sets of initial conditions and having particular interest in the applications, see \cite{Y} for a survey on
the subject.}. The knowledge of the invariant measure of interest, gives information on
the statistical behavior for the long time evolution of the system. This fact
strongly motivates the search for algorithms which are able to compute
quantitative information about invariant measures of physical interest, and
in particular, algorithms giving an explicit bound on the error which is
made in the approximation.
The application of such rigorously certified estimates allows to get reliable information on the statistical behaviour of the system and perform computer-aided proofs, establishing rigorously proved statements on the statistical behavior of the system (see e.g. \cite{GMN1}).

\paragraph{\textbf{Several levels of precision in the estimation of the approximation
error. }}The problem of approximating some interesting invariant measure of a
deterministic or random dynamical system is widely studied in the
literature. Some algorithms are proved to converge to the real invariant
measure (up to errors in some given metrics) in some classes of systems.
Sometimes asymptotical estimates on the rate of convergence are provided (see
e.g. \cite{Din93, Din94}, \cite{DelJu02,Del99} ,\cite{BM}, \cite{Mr}, \cite%
{F08},\cite{FC}); other results and algorithms give an explicit bound on the
error (see e.g. \cite{KMY,BB,L,PJ99,H,W,GMN1,GMN2}). This is the point of
view of the present paper. 

We are not only interested to the algorithm but
also to a suitable implementation. In fact, implementing such an algorithm
in a software which is able to keep track of the various truncations and
numerical errors in the computation allows the result of a single computation
to be interpreted as a computer-aided proved statement on the behavior of
the observed system, and hence it has a mathematical meaning. In the literature
the dimension of some nontrivial attractors or repellers was estimated in
this way (see e.g. \cite{GaNi}, \cite{GMN2}, \cite{PJ}), as well as escape rates (\cite%
{GNS}), linear response (\cite{BGNN1}\cite{PV2}), diffusion coefficients (%
\cite{BGNN2}\cite{PV}) or the behavior of Lyapuov exponents in models of
real phenomena (\cite{MSDGG},\cite{GMN1}).

It is worth noting that some negative result are known about the general problem
 of computing invariant measures up to a small given error. In
\cite{GalHoyRoj3} it is shown that\emph{\ there are examples of computable%
\footnote{%
Computable, here means that the dynamics can be approximated at any accuracy
by an algorithm, see e.g. \cite{GalHoyRoj3} for precise definition.} systems
without any computable invariant measure}. This phenomenon shows that there is some subtlety in the
general problem of computing invariant measures up to a given error.

\paragraph{\textbf{Finite element reductions based on a projection and the present
paper.}} \ The techniques used in the literature to establish rigorous bounds on the
approximation error are often related to a suitable finite-element reduction
of the transfer operator of the system. In this approach the transfer
operator is approximated by a finite-rank one. The invariant measures of
the system under study can be seen as fixed points of its transfer
operator when acting on suitable functional spaces.
 These fixed points can then be approximated by the fixed points of the finite-dimensional reduction of the
operator. Suitable quantitative fixed-point stability results can give a
bound of this approximation error.

For this purpose, several approaches have been implemented. The Ulam method
is a classical example of such a finite elements reduction, and provides a
finite dimensional approximation of the transfer operator with a finite Markov
chain obtained by discretizing the space by a cell subdivision; see Section \ref{sec:Ulam} for a precise definition.
In this approach, and in other finite-element reductions, the transfer operator is approximated by a finite-dimensional 
operator defined by the composition of the original operator with suitable projections to a
finite-dimensional functional space.
 In the classical Ulam method, a probability density is approximated by a piecewise constant one and the
projection is then a conditional expectation made on the cell subdivision of
the whole space. Other approaches use different approximation schemes, as
for instance a piecewise linear approximation (see Section \ref{sec:PL}), piecewise
smooth approximations, or even other approximation schemes based on Fourier
analysis or Taylor series (\cite{W},\cite{BS}), which are suitable for smooth
systems. All of these approaches require their own estimates and have
advantages for certain classes of systems: for instance, approximation schemes based on the projection to spaces of smooth functions converge faster when used to approximate smooth systems with smooth invariant measures.
 These approaches can be seen as examples of a general construction in which one defines a finite-dimensional
reduction of some operator by composing it with a suitable finite dimensional projection.

In this paper we consider this ``projection based'' finite-element reduction point view in general, and
show that if the finite element reduction method satisfies a certain list of
hypotheses, then we can apply a general construction in which the computation of the invariant density up to a small explicit approximation error will work efficiently.

To estimate this approximation error, we will consider the finite element reduction of the system as a small perturbation of the system itself and estimate quantitatively  the stability of the invariant measure of a system up this small perturbation. These kinds of estimates are also called quantitative statistical stability estimates.
It is known that the quantitative statistical stability  of a system is related to  the speed of convergence to equilibrium of the system itself: the faster is this speed of convergence, the more the system is statistically stable (see e.g. \cite{GJEP} for a general statement adapted to many convergence rates)\footnote{\label{note1} In our paper we will consider the transfer operator associated to the system acting on different weaker or stronger spaces with norms $||\ ||_w , ||\ ||_s$. Here by speed of convergence to equilibrium we mean the rate of convergence to the invariant measure $\mu$ of iterates $L^n\nu$ of regular initial probability measures $\nu$ by the transfer operator $L$. The speed of convergence to equilibrium will be measured as the speed of convergence to $0$ of the ratio $\frac{||L^n\nu-\mu||_w }{||\nu||_s}$. This notion is also related to the speed of mixing of the system.}. 
This is  a delicate point in many papers related to rigorous computations of invariant measures, where the estimate for approximation error involves an estimate for the convergence to equilibrium of the system. Establishing an effective (not only asymptotical) estimate for the convergence to equilibrium of the system is not trivial. This problem is sometimes approached by a-priori estimates which are possible only on restricted families of systems. For example, in circle expanding maps such explicit estimates on the convergence to equilibrium can be done by using Hilbert cones related techniques.
In \cite{GaNi} an idea to overcome this difficulty was proposed, and in this paper a construction is shown, in which the \emph{a priori} estimate on the speed of convergence is replaced by some \emph{a posteriori} one which is computed on the finite element reduction of the system. This is a finite dimensional system (and the transfer operator can be represented by a large and sparse matrix) and its speed of convergence to equilibrium can be estimated directly by the computer.
This idea allowed \cite{GaNi} to compute with explicit error bounds invariant densities of quite different systems as expanding maps, piecewise-expanding ones without a Markov partition and even non-uniformly expanding ones (examples of Manneville--Pomeau maps), essentially applying the same construction for each one of these systems.  An estimate of the convergence rate of a finite-dimensional system, as we need in the ``a posteriori'' approach, is always possible, but it can be a challenging task when the related matrix is large.

In \cite{GNS}, a method to speed up this computation was proposed and applied to some class of examples. This method exploits the regularization properties of the transfer operator to infer the speed of convergence to equilibrium of a finite-dimensional reduction of the system from a coarser finite-element reduction (hence reducing the dimension of the matrix to be considered when estimating the speed of convergence to equilibrium). We will refer to this kind of approach as a  ``coarse-fine'' approach.
In \cite{GMN1}, a similar approach was applied to estimate the convergence to equilibrium of high-resolution finite-element reductions (the rank of the reduced operator is of the order of millions) of transfer operators related to a class of random systems which are models of the behavior of the famous Belosouv-Zhabotisky chaotic chemical reaction, proving the existence of a noise-induced phenomenon observed by numerical simulation in 1983 in the article \cite{MT}.

In the present paper we propose a general systematic formalization of this method, adapting it to different kinds of projection based finite dimensional reductions.
We also implemented these ideas in the Julia language~\cite{Julia}, in a package called \texttt{RigorousInvariantMeasures.jl}, which is part of the JuliaDynamics organization. 
The package can be installed through the Julia package manager and the source code for the development version
can be found at 
\begin{center}
    \url{https://github.com/JuliaDynamics/RigorousInvariantMeasures.jl}.
\end{center}
Examples of the use of this package can be found in the \texttt{examples} directory.
Jupyter notebooks detailing its usage were developed for a summer school at Hokkaido University 
and can be found at \url{https://github.com/orkolorko/HokkaidoSchool}; Lectures 1 and 2 are introductory
while Lecture 3 and 4 deal with the rigorous approximation of the invariant density for a 
deterministic dynamical system and a random dynamical system respectively.

We will apply this new package to a series of examples already studied in \cite{GaNi}, 
testing sistematically the performance of the computations and showing a major speed-up and increase of precision with the new package.
 
\noindent {\bf Structure of the paper and main results. }
In Section \ref{sec:2} we describe the properties we require for our general projection based approximation schemes and the kind of operators to which we mean to apply it. We also show the first useful consequences of these properties, as the fact that the if the original transfer operator satisfy a Lasota Yorke inequality, also the finite element reduction of the transfer operator satisfies it.
In Section \ref{sec:3} we show  explicit bounds on the approximation errors made on approximating the fixed points of the original operator with the fixed points of the finite dimensional reduction.
In Section \ref{sec:coarse-fine} we show how to improve this bound and related estimates on the convergence to equilibrium by a coarse-fine strategy, in which we discretize the transfer operator at different resolutions, exploiting the regularization properties of the operator and using information from the coarser discretization to understand the behavior of the finer one, greatly improving the efficiency of the computation.

%{\em  In this section we also show that under the set of hypothesis we consider in the abstract setting defined at \ref{sec:2}  our algorithm works, in the sense that the  approximation error on the invariant measure can be made as small as wanted by increasing the resolution of the approximation. (questo forse va levato)}

In Sections~\ref{sec:Ulam}  and~\ref{sec:PL} we show two examples of approximation schemes, with associated functional analytic setting satisfying the abstract approximation setting defined at \ref{sec:2} : the Ulam scheme and a smoother approximation scheme based on the approximation by piecewise linear functions.

In Section~\ref{sec:8} we discuss some algorithmic aspects of the implementation of our ideas, in particular about the construction of the discretized operators and the estimation of norms of powers of discretized operators.
Section \ref{sec:9} presents examples and in Section \ref{sec:10} we present some final discussion and remarks.

\begin{notation}\label{not:general}
In the following, we use $I$ for the identity matrix/operator/function
(it is typically clear from the context which one it is), and $e_j$ for the
$j$-th vector of the canonical basis (i.e., the $j$th column of $I$).

The symbol $v^*$ denotes the conjugate transpose of a vector.

The symbol $\norm{f}_{L^p}$ denotes the $L^p$ norm of a function (usually defined on $[0,1]$), whereas the symbol $\norm{v}_{\ell^p}$ denotes the $\ell^p$ norm of a vector $v\in\mathbb{R}^n$.
\end{notation}

\section{The abstract setting}\label{sec:2}
In the following we will consider suitable operators between normed vector spaces
of functions over a certain compact manifold with boundary $X$; the main example 
is the transfer operator of a nonsingular dynamical systems on $X$, see Section 
\ref{sec:transferoperator}. We will suppose $X$ to be endowed with the normalized Lebesgue measure $m$ as a reference measure. And denote by $||. ||_{L^1}$ the norm of the associated space  $L^1(X,m)$.

\begin{hypothesisspace}\label{assumption:space}
Let $(\mathcal{U}_w, \norm{.})$ be a real or complex Banach space  of real or complex functions over $X$ containing the indicatrix  $1$ of the whole space. Let $\mathcal{U}_s\subseteq \mathcal{U}_w$  be a subspace of more regular functions on which a certain seminorm  $\norm{.}_s$ is defined. Let us suppose that $(\mathcal{U}_s, \norm{.}_s+\norm{.})$ is a Banach space which is compactly embedded in $(\mathcal{U}_w, \norm{.})$.

We will suppose that these norms  satisfy the following assumptions, there exists positive constants $S_1,S_2\in \mathbb{R}$ and an element $i\in\mathcal{U}_s^*$ such that:
\begin{enumerate}
 \item $\norm{1}=1$,
 \item $\norm{1}_s=0$,\label{ass:strong_norm_1}
 \item $\abs{i(f)} \leq \norm{f}$,
 \item $\norm{.}_{L^1} \leq \norm{.}$,
 \item $\norm{.}\leq S_1 \norm{.}_s+S_2\norm{.}_{L^1}$,
 \item i(1)=1.
\end{enumerate} 
\end{hypothesisspace}

\begin{example}\label{ex:ulam}
Let $\mathcal{U}_s=BV([0,1])$, the space of functions of bounded variation on $[0,1]$, equipped with the seminorm
$\norm{.}_s=\Var(.)$ and the norm $\norm{.}=\trinorm{.}$, with $S_1=1, S_2=1$ and
\[
i(f)=\int_0^1 f\, dm,
\]
where $m$ is the Lebesgue measure on $[0,1]$.
\end{example}
\begin{example}\label{ex:lipschitz}
Let $\mathcal{U}_s=\Lip([0,1])$, the space of Lipschitz continuous functions on $[0,1]$, equipped with the seminorm
$\norm{.}_s=\Lip(.)$ and the norms $\norm{.}=\norm{.}_{\infty}$, with $S_1=1, S_2=1$ and
\[
i(f)=\int_0^1 f \, dm,
\]
where $m$ is the Lebesgue measure on $[0,1]$.
\end{example}

\begin{hypothesisoperator}\label{ass:5}
Let $L$ be an operator acting on $\mathcal{U}_s$ such that
\begin{itemize}
\item $i(Lf)=i(f)$,
\item $\norm{L} < \infty$,
\end{itemize}
and suppose that there are $A,B,W\in \mathbb{R}$, with $A<1$ such that for each $n\geq 0 $
\begin{eqnarray} 
\label{eq1}||Lf||_1\leq ||f||_1 \\
\label{genericLY}	\norm{Lf}_s \leq A \norm{f}_s + B\trinorm{f} \\
\label{eq3}	||L^n f||\leq W||f|| 
\end{eqnarray}
respectively for all functions  in $L^1$ and in $ \mathcal{U}_s$.
We say such an operator satisfies a \textbf{one step Lasota-Yorke inequality}.
\end{hypothesisoperator}

\begin{remark} The Lasota Yorke inequality  implies  that $L$ has a `regularizing' behavior, up to a certain point.
The general form of the Lasota Yorke inequality is the following: there are $\lambda <1$, $A',B\geq 0$ s.t. for each  $f \in \mathcal{U}_s$  and $n\geq 0$
\begin{equation} \label{generalLY}
	\norm{L^nf}_s \leq A' \lambda^n \norm{f}_s + B\trinorm{f}.
\end{equation}
An estimate of this kind can be estabished in many systems having some form of uniform expansiveness, even in the presence of discontinuities or piecewise hyperbolic behavior.
We remark that in this case a suitable iterate of $L$ satisfies \eqref{genericLY}. This is usually sufficient for the computation of invariant densities of a system, as the invariant density of the original system is also invariant for the iterate.

We remark that since $\norm{f}\geq \norm{f}_{L^1}$, \eqref{genericLY} also implies  
\begin{equation} \label{LYsw}
	\norm{Lf}_s \leq A \norm{f}_s + B\norm{f}
\end{equation}
for all functions $f \in \mathcal{U}_s$.
\end{remark}

A consequence of \eqref{genericLY} is a simple regularity estimate on a the fixed points  of $L$
which will play an important role in our estimation procedure.
\begin{corollary} \label{cor:LYstrongbound}
If $L$ satisfies a one step Lasota-Yorke inequality and $u$ is a fixed point of $L$:
\begin{equation} \label{LYstrongbound}
	\norm{u}_s \leq \frac{B}{1-A}\trinorm{u}.
\end{equation}
\end{corollary}

%\begin{remark}
%The estimate in Corollary \ref{cor:LYstrongbound}, is not yet fully explicit, because we have no way to bound $\trinorm{u}$, in full generality, for the invariant measure $u$ normalized with  $i(u)=1$.
%Finding norms (and associated Banach spaces) that allow us to prove the fact that a fixed point belongs to a given Banach space is an important problem in the functional analytic treatment of dynamical systems (citare Anisotropic, etc...)
%We think that the treatment in the current paper, with the flexibility of adding an auxiliary norm, may ease this task in general situations.\todo{a Stefano il senso di questi remarks non è chiaro}
%\end{remark}

We are interested to compute the invariant density by a suitable finite element reduction of our system.
This finite element reduction will be realized by a  suitable projection on a finite dimensional space.
We now formalize the requirements we ask for this projection.
\begin{definition}\label{def:comp_discretization}
Let  $n\in{\mathbb N}$, and  $P_h$ be a rank $n$ linear operator defined on $\mathcal{U}_s$, with $h:=1/n$.

We say that $P_h$ is a \textbf{compatible discretization} if there exists $K$ and $E$ such that:
\begin{enumerate}
 	\item $P_h = P_h^2$, i.e., $P_h$ is a projection.
 	\item  $\norm{P_h f} \leq \norm{f}$ for any function $f\in \mathcal{U}_s$. \label{Phprop:norm1}
 	\item $\norm{P_h f}_s \leq \norm{f}_s$ for any function $f\in \mathcal{U}_s$. \label{Phprop:strong_norm_contraction}
 	\item $\norm{P_h f - f} \leq K h \norm{f}_s$. \label{Phprop:normapprox}
 	\item $\trinorm{P_h f}\leq \trinorm{f}+E h \norm{f}_s$ \label{Phprop:boundweak}
 	\item $\trinorm{P_h f + i(f - P_h f) } \leq \trinorm{f} +E h \norm{f}_s$. \label{Phprop:trinormapprox}
\end{enumerate}
\end{definition}

\begin{remark}\label{rem:linear_h}
In general, we could relax Items \ref{Phprop:normapprox},\ref{Phprop:boundweak} \ref{Phprop:trinormapprox}, substituting $h$ by $h^{\alpha}$, with $\alpha>0$
in the whole paper, or more generally substitute $K\cdot h$ and $E\cdot h$ by functions $K(h)$ and $E(h)$ that go to $0$ fast enough as $h$ goes to $0$; this is not needed for the projections and functional spaces we study in this paper, but most of the theory adapts to these more general conditions with few differences.
\end{remark}
\begin{remark}
Condition \ref{Phprop:trinormapprox} in Definition \ref{def:comp_discretization} is used to control the error when the discretization of the operator is not $i$-preserving (see Definition \ref{def:discretized_operators} and Remark \ref{rem:ipreserv}).
\end{remark}
        
\begin{definition}
We will call the finite dimensional space $\mathcal{U}_h:=P_h(\mathcal{U}_s)$ the
\textbf{approximating space}.

The strong norm, the weak norm and the $L^1$ norm induce norms on $\mathcal{U}_h$,
that will use the same notation.
\end{definition}

\begin{hypothesisapproximatingspace}
We assume that the norms on $\mathcal{U}_h$ satisfy the following inequality; there exist 
$M$ and $\alpha$ such that for each $f\in \mathcal{U}_h$
\begin{equation}\label{NS}
    \norm{f}_s \leq \frac{1}{h^{\alpha}} M\norm{f}.
\end{equation}
\end{hypothesisapproximatingspace}
\begin{remark}
Such an inequality is usually false on $\mathcal{U}_s$, but the approximating space 
$\mathcal{U}_h$ is finite dimensional. If we let  
$\norm{f}_{ns} = \norm{f}_s+\norm{f}$, this is a norm on $\mathcal{U}_h$ and 
there exists constants
$\Lambda(h), \Theta(h)$, depending on $h$,  such that 
\[
\Lambda(h)\norm{f} \leq \norm{f}_{ns} \leq \Omega(h) \norm{f};
\]
remark that as $h = 1/n$ goes to $0$, $\Omega(h)$ may go to infinity. 

The rate at which $\Omega(h)$ goes to infinity depends on the chosen norms 
and approximation schemes.
For the schemes presented in the current paper $\Theta(h) = M/h$; 
this is not true in general, in other cases, as in Chebyshev and Fourier 
discretization $\Theta(h)$
may grow faster as shown by Markov-Bernstein inequalities \cite{Milovanovic}; 
as an example, suppose  $\mathcal{U}_s = C^1([0,1])$, $\norm{f}_s = \norm{f'}_{\infty}$, 
$\norm{f} = \norm{f}_{\infty}$ and $P_n$ is the map that associates 
to $f$ its Chebyshev interpolant of degree $n$.

In this case, by $\norm{f-P_n f}\leq O(h)\norm{f}_s$ but, by Markov-Bernstein 
\[
    \norm{p'}_{\infty}\leq n^2 \norm{p}_{\infty}, 
\]
for all $p$ polynomial of degree at most $n$, i.e., $\Theta(h)=1/h^2$. 
\end{remark}

\begin{remark}
Another possible generalization is to allow 
\[
||f||_{L^1}\leq \tilde{M}||f||, 
\]
instead of fixing $\tilde{M}$ to be equal to $1$ as in Assumption \ref{assumption:space}.
Again, this is not needed in our current paper, but our methods can be adapted to this case.
\end{remark}

\begin{definition}\label{def:discretized_operators}
Given a compatible discretization $P_h$ we define the
\textbf{discretized operator} to be
\[
L_h:\mathcal{U}_h\to \mathcal{U}_h, \quad L_h := P_h L P_h,
\]
and we define the \textbf{$i$-preserving discretized operator} to be
\begin{equation} \label{defQh}
    Q_h:\mathcal{U}_h\to \mathcal{U}_h, \quad Q_h f := L_h f + 1\cdot(i(f)-i(L_h f)).
\end{equation}
\end{definition}
\begin{remark}\label{rem:ipreserv}
In the two explicit discretizations presented in this paper, $i$ is the integral with respect to the Lebesgue measure. The name $i$-preserving may be interpreted, in these discretizations, as a nickname for integral preserving.
\end{remark}

\begin{remark}
Depending on the chosen compatible discretization, $L_h$ may preserve $i$.
In this case $Q_h$ and $L_h$ are going to denote the same operator.
\end{remark}

\begin{remark}
From assumption \ref{assumption:space} item \eqref{ass:strong_norm_1}, follows that:
\begin{equation}
\norm{Q_h f}_s \leq \norm{L_h f}_s. \label{Phprop:freenormalization}
\end{equation}

\end{remark}
From the properties of a compatible discretization follows a straightforward result
on the operators $L_h$ and $Q_h$.

\begin{corollary} \label{cor:discLY}
Let $P_h$ be a compatible discretization, and suppose that
$L$ satisfies a one step Lasota-Yorke inequality~\eqref{genericLY} with coefficients $A$ and $B$. %which satisfies items~\ref{Phprop:strong_norm_contraction} and~\ref{Phprop:freenormalization} of Lemma~\ref{lem:Phprop}.
Then, if $h$ is small enough, a one step Lasota-Yorke inequality holds for $L_h$ and $Q_h$:
for all $f\in \mathcal{U}_s$
\begin{equation}\label{eq:discLY_gen}
\norm{Q_h f}_s\leq \norm{L_h f}_s\leq (A+EhB)\norm{f}_s+B\trinorm{f}.
\end{equation}
Moreover, for all $f_h\in \mathcal{U}_h$, we have a stronger one step Lasota-Yorke inequality, since $P_h f_h=f_h$:
\begin{equation}\label{eq:discLY_disc}
\norm{Q_h f_h}_s\leq\norm{L_h f_h}_s\leq A\norm{f_h}_s+B\trinorm{f_h}.
\end{equation}
\end{corollary}
\begin{proof}
For $f\in \mathcal{U}_s$ and the properties of a compatible discretization we have that
\begin{align*}
\norm{L_h f}_s&=\norm{P_h L P_h f}_s\leq \norm{L P_h f}_s\leq A\norm{P_h f}_s+B\trinorm{P_h f}\\
&\leq A\norm{f}_s+B(Eh\norm{f}_s+\trinorm{f});
\end{align*}
if $f_h\in \mathcal{U}_h$ we have that $P_h f_h=f_h$, from this follows:
\[
\norm{L_h f_h}_s=\norm{P_h L P_h f_h}_s=\norm{P_h Lf_h}_s\leq A\norm{f_h}_s+B\trinorm{f_h}.
\]
\end{proof}

\begin{corollary} \label{cor:KeLi}
Applying repeatedly the Lasota-Yorke inequality of Corollary~\ref{cor:discLY}, we get for all $k\in\mathbb{N}$ and $f_h\in\mathcal{U}_h$
\begin{align*}
\norm{Q_h^k f_h}_s \leq A^k \norm{f_h}_s + \frac{B}{1-A}\max_i(\trinorm{Q^i_h}) \trinorm{f_h}.
\end{align*}

If $f\in\mathcal{U}_s$ we have that $Q_h f\in\mathcal{U}_h$, therefore
\begin{align*}
\norm{Q_h^k f}_s \leq A^{k-1}(A+E h B) \norm{f}_s + \frac{B}{1-A}\max_i(\trinorm{Q^i_h}) \trinorm{f}.
\end{align*}
\end{corollary}
\begin{proof}
If $f_h\in\mathcal{U}_h$ then
\begin{align*}
\norm{Q_h^k f_h}_s &\leq A \norm {Q_h^{k-1}f_h}_s + B\trinorm{Q_h^{k-1}f_h} \\
& \leq A (A\norm {Q_h^{k-2}f_h}_s + B\trinorm{Q_h^{k-2}f_h}) + B\trinorm{Q_h^{k-1}f_h}\\
& \leq \dots \leq A^k \norm{f_h}_s + B\sum_{j=1}^{k-1} A^{k-1-j}\trinorm{Q_h^j f_h}\\
& \leq A^k \norm{f_h}_s + \frac{B}{1-A}\max_i(\trinorm{Q^i_h}) \trinorm{f_h}.
\end{align*}
If $f\in\mathcal{U}_s$ we have that $Q_h f\in\mathcal{U}_h$, therefore
\begin{align*}
\norm{Q_h^k f}_s &= \norm{Q_h^{k-1} Q_h f}_s\\
&\leq A^{k-1} \norm{Q_h f}_s + B\sum_{j=1}^{k-2} A^{k-1-j}\trinorm{Q_h^{j+1} f}\\
&\leq A^{k-1}(A+E h B) \norm{f}_s+ B\sum_{j=1}^{k-1} A^{k-1-j}\trinorm{Q_h^{j+1} f}\\
& \leq A^{k-1}(A+E h B) \norm{f}_s + \frac{B}{1-A}\max_i(\trinorm{Q^i_h}) \trinorm{f}.
\end{align*}
\end{proof}

\begin{remark}
Remark that if the Lasota-Yorke inequality \eqref{eq:discLY_gen} is satisfied for a 
discretization of size $n = 1/h$ then for all discretizations with $\tilde{n}>n$
we have that $1/\tilde{n}=\tilde{h}<h$ and so inequality \eqref{eq:discLY_gen} is 
satisfied for all finer discretizations and for the original operator $L$.
This permits us to prove, in a similar fashion as Corollary \ref{cor:KeLi} that we have a uniform
iterated Lasota-Yorke inequality.
This is the main hypothesis we need to satisfy so that the spectral stability results of \cite{KeLi} holds.
\end{remark}

\section{Fixed point error estimation}\label{sec:3}
In this section, we describe an explicit strategy to derive certified approximations of the fixed point $u$ of $L$ by approximating it with an element
$u_h$ of the approximating space.

The following theorems give a slightly improved version of \cite[Theorem~3.1]{GaNi} in which we allow for an inexactly computed eigenvector $u_h$ and we take more care about the multiplicative factors $\norm{P_h}$.

\begin{definition}
Let us consider the generalized ``zero average'' spaces 
$$\mathcal{U}_w^0:=\{v\in \mathcal{U}_w \mid i(v)=0\},$$  
$$\mathcal{U}_s^0:=\{v\in \mathcal{U}_s \mid i(v)=0\}.$$
When dealing with the discretized operator $Q_h$, we denote by
$$\mathcal{U}_h^0:=\mathcal{U}_s^0\cap \mathcal{U}_h.$$
\end{definition}
\begin{remark}
When restricted to $\mathcal{U}_h$, the strong and the weak norm are equivalent, therefore
$\mathcal{U}_h^0$ can be equivalently defined as
\[
\mathcal{U}_h^0:=\mathcal{U}_w^0\cap \mathcal{U}_h.
\] 
\end{remark}
\begin{remark}
In Examples \ref{ex:ulam}, \ref{ex:lipschitz}, the space $\mathcal{U}_w^0$ is the space of average $0$ functions, i.e.,
\[
\mathcal{U}_w^0=\mathcal{U}_w\cap \{f \in L^1([0,1]) \mid \int f dm=0\}.
\]
Then
\[
\mathcal{U}_h^0=\{u_h \in \mathcal{U}_h \mid h \sum_i u_i\},
\]
where $u_i$ are the coordinates of $u_h$.
\end{remark}

\begin{theorem}\label{th:fixed_point}
In the framework of the assumptions on the spaces, operators and discretizations stated in Section \ref{sec:2}, Let $L$ be an operator operator with a fixed point $u$, normalized in a way that  $i(u)  = 1$, let $Q_h$ be an $i$-preserving discretized operator,
and let $u_h \in \mathcal{U}_h$ be any vector such that $\norm{Q_hu_h - u_h} \leq \varepsilon$, normalized so that $i(u) = i(u_h) = 1$.
Let $C_k$, for each $k\in\mathbb{N}$, be a constant such that
\begin{equation} \label{Ckbound}
  \norm{Q_h^k|_{\mathcal{U}_h^0}} \leq C_k, \quad k \in\mathbb{N}  
\end{equation}
and suppose that $\sum_{k=0}^{\infty} C_k<\infty$.
Then,
\begin{equation} \label{theorem1_estimate}
\norm{u-u_h} \leq \left(\sum_{k=0}^{\infty} C_k\right)(2 K h \left(1+\norm{L}\right)\norm{u}_s + \varepsilon).
\end{equation}
\end{theorem}

Before the proof of the theorem we need to perform some technical estimates.

\begin{lemma} \label{lemma:discbound}
Let $L: \mathcal{U}_s \to \mathcal{U}_s$ be an $i$-preserving operator,
$L_h$ and $Q_h f$ discretized operators as in Definition \ref{def:discretized_operators} obtained by a compatible discretization as in Definition \ref{def:comp_discretization}.
Then,
\[
	\norm{Q_h f - Lf} \leq 2 K h \left(\norm{L}\norm{f}_s + \norm{Lf}_s\right).
\]
\end{lemma}
\begin{proof}
The inequality follows by combining
\begin{align*}
\norm{Q_h f - Lf} &= \norm{L_h f - Lf} + \norm{1 \cdot  (i(Lf) - i(L_hf))}\\
 & \leq 2\norm{L_h f - Lf}
\end{align*}
where we used \ref{assumption:space}, item 1 and 3, and
\begin{align*}
	\norm{L_h f - Lf} &\leq \norm{P_h L (P_h - I)f} + \norm{(P_h-I)Lf}\\
	& \leq \norm{L}K h \norm{f}_s + K h \norm{Lf}_s.
	\qedhere
\end{align*}
\end{proof}

\begin{corollary}\label{coro}
If $u\in\mathcal{U}_s$ is a fixed point of the operator $L$, then
\[
\norm{Q_h u - u} \leq 2 K h \left(\norm{L}+1\right)\norm{u}_s.
\]
\end{corollary}
\begin{proof}
By the previous lemma
\[
\norm{Q_h u - u} = \norm{Q_h u - Lu} \leq 2 K h \left(\norm{L}\norm{u}_s + \norm{Lu}_s\right), 
\]
observing that $\norm{Lu}_s = \norm{u}_s$ we have the thesis.
\end{proof}

\begin{proof}[Proof of Theorem \ref{th:fixed_point}]
Let $v = u - u_h \in \mathcal{U}_h^0$. Note that from Corollary \ref{coro} we get 
\[
\norm{Q_h v - v} \leq \norm{Q_h u - u} + \norm{Q_h u_h - u_h} \leq 2 K h \left(1+\norm{L}\right)\norm{u}_s + \varepsilon.
\]
By the triangle inequality,
\[
\norm{v} \leq \norm{Q_h v-v} + \norm{Q_h^2 v - Q_h v} + \norm{Q_h^3 v - Q_h^2 v} + \dots + \norm{Q_h^m v - Q_h^{m-1} v} + \norm{Q_h^m v},
\]
and since $\lim_{m\to\infty} \norm{Q_h^m v} \leq \lim_{m\to\infty} C_m \norm{v} = 0$ we can take the limit obtaining
\begin{equation} \label{boundv}
	\norm{v} \leq  \sum_{k=0}^{\infty}\norm{Q_h^k (Q_h v-v)}  \leq \left(\sum_{k=0}^{\infty} C_k\right) \norm{Q_h v-v}.
\end{equation}
Combining these inequalities we get:
\[
\norm{v} \leq \left(\sum_{k=0}^{\infty} C_k\right) (2 K h \left(1+\norm{L}\right)\norm{u}_s + \varepsilon).
\]
\end{proof}

While Theorem~\ref{th:fixed_point} requires an infinite sum, the following lemma shows that it is sufficient to find a value $m$ with $C_m < 1$ to prove the convergence of the series.
\begin{lemma} \label{lem:onembound}
Let $Q_h$ be an $i$-preserving discretized operator, and $C_k$ be constants such that   $\norm{Q_h^k|_{\mathcal{U}_{h}^0}} \leq C_k$ for each $k=0,1,2,\dots$. Suppose that $C_m < 1$ for some positive integer $m$. Then,
\begin{enumerate}
  \item $\sum_{k=0}^{\infty} C_k \leq \frac{1}{1-C_m}(C_0+C_1+\dots+C_{m-1}) < \infty$;
  \item there are real constants $C > 0, \lambda_2 \in (0,1)$ such that $\norm{Q_h^k|_{\mathcal{U}_{h}^0}} \leq C \lambda_2^k$ for each $k$.
\end{enumerate}
\end{lemma}
\begin{proof}
Let $k\in\mathbb{N}$, and use Euclidean division with remainder to write $k = qm +r$. In particular, we have $r\in \{0,1,\dots,m-1\}$ and $k < (q+1)m$.

Since $Q_h$ is $i$-preserving, $Q_h(\mathcal{U}_h^0) \subseteq \mathcal{U}_h^0$, hence we can write
\[
\norm{Q_h^k|_{\mathcal{U}_{h}^0}} \leq \norm{(Q_h^m|_{\mathcal{U}_{h}^0})^q (Q_h^r|_{\mathcal{U}_{h}^0})} \leq C_m^q C_r.
\]
Then the first estimate follows by summing over all possible $k$
\begin{align*}
  \sum_{k=0}^{\infty} C_k &\leq C_m^0(C_0 + C_1+\dots + C_{m-1}) + C_m^1(C_0 + C_1+\dots + C_{m-1})\\
  &  \qquad + C_m^2(C_0 + C_1+\dots + C_{m-1}) + \dots \\
  & \leq (1+C_m + C_m^2 + \dots )(C_0 + C_1+\dots + C_{m-1}).
\end{align*}
The second estimate follows instead from noting that 
\[
C_m^q C_r \leq (C_m)^{\frac{k}{m}-1} \max(C_0,C_1,\dots,C_{m-1}),
\]
and thus we can take
\[
C = \frac{\max(C_0,C_1,\dots,C_{m-1})}{C_m}, \quad \lambda_2 = (C_m)^{\frac1m}.
\]
\end{proof}
The first estimate is tighter and is the one that we shall use in numerical computation; the second one is looser but it gives an explicit bound with a geometric series.

\begin{remark} \label{remconv}
The sequence  $\norm{Q_h^k|_{\mathcal{U}_{h}^0}}$ is related to the speed of convergence to equilibrium of the system mentioned in note \ref{note1}. Even if these norms are explicitly computable, since $Q^k_h$ is a finite rank operator and can be represented by a matrix, we are not going to compute
an enclosure for the norm, but just an upper bound $C_k$, which is enough for our treatment and more practical to compute. 

In the case of Markov Transfer operators, this sequence is also related to the convergence of equilibrium of the system, indeed if $\mu$ is invariant for the system and $\nu $ is another probability measure in the strong space we have that $\mu-\nu \in {\mathcal{U}_s^0}$ and hence the convergence  to zero in the weak norm of $Q^k_h(\mu-\nu)$ can be estimated by the  sequence $C_k$. Note also that $Q^k_h(\mu-\nu)=\mu -Q^k_h(\nu)$.
\end{remark}

\begin{remark}
In Theorem  \ref{th:fixed_point} we have a summability condition on $C_k$. We remark that in the statement and in the proof of the theorem
we could exchange the role of $L$ and $Q_h$. If
we could prove that $\sum_k \norm{L^k|_{\mathcal{U}^0_s}}$ is summable and find an estimate for each term, this would give us an \emph{a-priori} bound on the approximation error, but in general this a difficult task already for simple maps, as one-dimensional piecewise expanding ones, in the case there is not a Markov partition.

The flexibility of our method lies in the fact that the bound in Theorem \ref{th:fixed_point} uses an a-posteriori, computer-assisted estimate which is computed on a finite-dimensional operator $Q_h$: in some sense, we ask the computer to estimate the convergence to equilibrium of the system at a finite resolution.
This task is possible even if the dynamics is quite complicated. Of course the complexity increases with the resolution, and to optimize this we have to find a suitable strategy. This is the theme of next section.
\end{remark}

% \begin{remark}
% In \cite{GaNi} (see Section 5.1) is proved that in the case of the transfer operator associated to piecewise expanding maps
%  and the Ulam discretization, this ``a posteriori'' method can be used to approximate the 
% fixed point up to any precision.

% \todo{S.G. Visto che l'enunciato del vecchio articolo è solo per le mappe PW expanding questo remark forse anddrebbe meglio nella sezione 5, doveci specializziamo a questi casi. 
% Se volete qui si potrebbe mettere una dimostrazione che dimostra che il metodo ``a posteriori'' funziona in generale con gli assiomi che abbiamo. La dimostrazione dovrebbe essere la stessa. Forse è questo che voleva fare Isaia all'inizio, ma a un certo punto non capivo cosa c'entrava con i conti che erano dopo nella versione precedente (che sono ora stati spostati).
% Per fare questa dimo generale, c'è un po' da lavorarci comunque.}
% \end{remark}
\subsection{The approximation error can be made as small as wanted} \label{sec:works}

Our error estimates are a-posteriori ones: one knows the quality of the approximation only after applying the algorithm. 
In this section we give an argument showing that if the spaces satisfy Assumptions \ref{assumption:space}, the discretization scheme satisfies
Definition \ref{def:comp_discretization} and the operator satisfies \ref{ass:5} we can approximate the stationary density as 
well as wanted; the argument here mirrors the one in \cite{GaNi} but works under the more general assumptions of this paper.

%In Section \ref{sec:coarse-fine}, Lemma \ref{lemma:exponentialCk} we will give an abstract argument that allows us to prove that it is possible 
%to estimate the mixing rates of the discretized operator on a finer discretization by the ones on a coarse discretization. 

Suppose $\mathcal{U}_{s}\subseteq \mathcal{U}_{w}$ are two vector spaces of Borel signed measures on a certain metric space $X$ 
endowed with two norms, the strong norm $||~||_{s}$ on $\mathcal{U}_{s}$ and the weak
norm $||~||$ on $\mathcal{U}_{w}$, such that $||~||_{s}\geq ||~||$ as before.
Let $\overline{\delta} \geq 0.$ Let $L_\delta $, $\delta \in [0,\overline{\delta} )$ be a family of Markov operators acting on $\mathcal{U}_{w}$. Denote  by $\mathcal{U}^0_{s}, \mathcal{U}^0_{w}$ the ``zero average'' spaces of  $\mathcal{U}_{s}, \mathcal{U}_{w}$.

\begin{definition}
We say that $L:\mathcal{U}_{s} \to \mathcal{U}_{s}$ has exponential convergence to  equilibrium if 
there are $\lambda<0$ and $C\geq 0$ such that for each $n\geq 0$, $f\in \mathcal{U}^0_s$ $$||L_0^n f|| \leq Ce^{\lambda n}||f||_s.$$. 
\end{definition}

\begin{theorem}\label{TAW} Let $L_0$ be an linear operator acting on $\mathcal{U}_{s}, \mathcal{U}_{w}$,
having exponential convergence to equilibrium, and let $L_h=P_h L_0 P_h$ where $P_h$ is a compatible discretization.
Let $\bar{h}$ be small and suppose that for all $h\in [0, \bar{h})$
\begin{enumerate}
    \item $L_h$ are Markov operators acting on $\mathcal{U}_w$ and $\mathcal{U}_s$,
    \item $L_h$ satisfy \eqref{LYsw} with constants uniform in $h$,
    \item $L_h$ satisfy \eqref{eq3} with constant uniform in $h$.
\end{enumerate}   
Then we can apply Theorem \ref{th:fixed_point}, finding constant $C_k$ such that when $h$ and $\varepsilon$ are small enough, $||u-u_h||$ in  \ref{theorem1_estimate} is as small as wanted.
\end{theorem}

Before the proof  we need to recall a result which is classical in this setting, and is proved in \cite{Gdisp} in the form we will use.

We say that $L_\delta$ is a uniform family of operators if:

\begin{enumerate}
\item[\textbf{UF1}] (Uniform Lasota Yorke ineq.) There are constants $%
A,B,\lambda _{1}\geq 0$ with $\lambda _{1}<1$ such that $\forall f\in
B_{s},\forall n\geq 1,\forall \delta \in \lbrack 0,1)$ and each operator
satisfies a Lasota Yorke inequality. 
\begin{equation}
||L_{\delta }^{n}f||_{s}\leq A\lambda _{1}^{n}||f||_{s}+B||f||_{w}.
\end{equation}

\item[\textbf{UF2}] Suppose that $L_{\delta }$ approximates $L_{0}$ when $%
\delta $ is small in the following sense: there is $C\in \mathbb{R}$ such
that $\forall g\in B_{s}$:%
\begin{equation}
||(L_{\delta }-L_{0})g||_{w}\leq \delta C||g||_{s}.
\end{equation}

\item[\textbf{UF3}] Suppose that $L_{0}$ has exponential convergence to
equilibrium, with respect to the norms $||~||_{w}$ and $||~||_{s}$.

\item[\textbf{UF4}] (The weak norm is not expanded) There is $M$ such that $%
\forall \delta ,n,g\in B_{s}$ $\ ||L_{\delta }^{n}g||_{w}\leq M||g||_{w}.$
\end{enumerate}

The following result (see \cite{Gdisp}, Proposition 45 for the proof) shows that such a uniform family has a  uniform rate of contraction of the space $\mathcal{U}^0_{s}$ and hence a uniform convergence to equilibrium and spectral gap.

\begin{theorem}[Uniform $\mathcal{U}^0_{s}$ contraction for the uniform family of
operators]
\label{unifcont}Let us consider a one parameter family of operators $
L_{\delta }$, $\delta \in \lbrack 0,1)$. Suppose that they satisfy
UF1,...UF4, then there are $\lambda _{1}<1$ and $A_{2},\delta _{2}\geq 0$
such that for each $\delta \leq \delta _{2}$ and $f\in V_{s}$%
\begin{equation}
||L_{\delta }^{k}f||_{s}\leq A_{2}\lambda _{1}^{k}||f||_{s}.
\label{contract3}
\end{equation}
\end{theorem}

\begin{proof}[Proof of Theorem \ref{TAW}]
First we see that we can apply Theorem \ref{unifcont} to our family of operatos $L_h$.
The assumption $UF1$ and $UF4$ are verified due to \eqref{genericLY} and \eqref{eq3}. The assumption $UF2$ is provided by Lemma \ref{lemma:discbound}, while $UF3$ is supposed in the assumptions of Theorem \ref{TAW}.
Applying Theorem \ref{unifcont} we get that uniformly on $h $ there are $\lambda_1<0,\overline{h}_2 \geq 0, C_1\geq 0$ such that for each $h\in [0,\overline{h}_2 ), n\geq 0$, $f\in  \mathcal{U}_s$, 
$$||L_h^n f||_s \leq C_1e^{\lambda _1 n}||f||_s.$$ 
By \eqref{NS} we than have that when $f\in U^0_h$
$$||L_h^n f||_w \leq ||L_h^n f||_s  \leq h^{-1}C_1e^{\lambda _1 n}M_1||f||_w.$$ 
By this we see that a sufficient  condition to get $C^n\leq \frac12$ is $$n\geq\lambda^{-1} log(h^{-1}C_1M_1)$$
by Item 1) of Lemma \ref{lem:onembound} this leads to $\sum_{k=0}^{\infty} C_k \leq 2 W \lambda^{-1} \log(h^{-1}C_1M_1) $ and then by \eqref{theorem1_estimate}:
\begin{equation}
\norm{u-u_h} \leq ( 2 W \lambda^{-1} \log(h^{-1}C_1M_1))(2 K h \left(1+\norm{L}\right)\norm{u}_s + \varepsilon).
\end{equation}
Which can be set as small as wanted when $h$ and $\varepsilon $ are small enough.
\end{proof}

\section{Estimating the convergence to equilibrium with the coarse-fine strategy}\label{sec:coarse-fine}

This section presents the coarse-fine approach, i.e., a method to use bounds $C_k$ as in~\eqref{Ckbound}, estimating the convergence to equilibrium of $Q_h$, to produce analogous bounds $C_k^F$ on the convergence to equilibrium of a finer-resolution approximation $Q_{h_F}$ of $L$, with $h_F < h$. An important ingredient wll be the Lasota-Yorke inequality, which is shared by all sufficiently fine compatible discretizations of $L$ (as proved in Corollary \ref {cor:discLY}).

A statement of this kind will be given in Corollary \ref{finalcor}.  This corollary will be obtained as a consequence of several intermediate steps, obtaining estimates on the norm of $Q_{h}^m  - Q_{h_F}^m$.

%We replicate the steps of Theorem~\ref{th:fixed_point},
%but generalizing the result so that it holds for all functions $f \in \mathcal{U}_{h_F}$
%instead that only for the invariant measure.

The first ingredient is an iterated version of the Lasota-Yorke inequality for a discretized operator (Corollary~\ref{cor:discLY}). The approach is somewhat similar to the one used in \cite{GNS} to rigorously estimate decay of correlation.

\begin{theorem}\label{thm:smallmatrix}
Let $L$ be an operator that satisfies assumption \ref{ass:5}, $P_h$ be a compatible discretization.
Then, for each $k\in\mathbb{N}$ we have the inequality
\[
\begin{bmatrix}
    \norm{Q_h^k f}_s\\
    \trinorm{Q_h^k f}
\end{bmatrix}
\leq
\left(
\begin{bmatrix}
    1 & 0\\
    E h & 1
\end{bmatrix}
\begin{bmatrix}
    A & B\\
    0 & 1
\end{bmatrix}
\right)^k
\begin{bmatrix}
    \norm{f}_s\\
    \trinorm{f}
\end{bmatrix},
\quad f \in \mathcal{U}_h,
\]
where $\leq$ is intended to be componentwise.
\end{theorem}
\begin{proof}
Note that $P_h f = f$ since $f\in\mathcal{U}_h$. We have
\[
\trinorm{Q_h f} \leq \trinorm{Lf} + E h \norm{Lf}_s \leq \trinorm{f} + E h \norm{Lf}_s.
\]
Hence
\begin{equation*}
\begin{bmatrix}
    \norm{Q_h f}_s\\
    \trinorm{Q_h f}
\end{bmatrix}
 \leq
\begin{bmatrix}
    1 & 0\\
    E h & 1
\end{bmatrix}
\begin{bmatrix}
    \norm{L f}_s\\
    \trinorm{f}
\end{bmatrix}
\leq
\begin{bmatrix}
    1 & 0\\
    E h & 1
\end{bmatrix}
\begin{bmatrix}
    A & B\\
    0 & 1
\end{bmatrix}
\begin{bmatrix}
    \norm{f}_s\\
    \trinorm{f}
\end{bmatrix}.
\end{equation*}
The rest follows by induction.
\end{proof}

\begin{corollary} \label{cor:Rkh}
Let $M$ be as in  \eqref{NS}.
Then,
\begin{equation}
	\norm{Q_h^k f}_s \leq R_{k,h,1} \norm{f}, \quad \trinorm{Q_h^k f} \leq R_{k,h,2} \norm{f}, \quad f \in \mathcal{U}_h,
\end{equation}
where
\begin{equation} \label{Rkhi}
\begin{bmatrix}
    R_{k,h,1}\\
    R_{k,h,2}
\end{bmatrix}
 := \left(
\begin{bmatrix}
    1 & 0\\
    E h & 1
\end{bmatrix}
\begin{bmatrix}
    A & B\\
    0 & 1
\end{bmatrix}
\right)^k \begin{bmatrix}
    \frac{1}{h^{\alpha}} M \\
    1
\end{bmatrix}.
\end{equation}
\end{corollary}

\begin{corollary}
Let $S_1,S_2$ be constants such that $\norm{f} \leq S_1 \norm{f}_s + S_2 \trinorm{f}$. Then,
\begin{equation} \label{RSbound}
	\norm{Q_{h}^k} \leq S_1 R_{k,h,1} + S_2 R_{k,h,2}.
\end{equation}
\end{corollary}

\begin{remark}
If $E=0$, as in the case of the Ulam projection (Section \ref{sec:Ulam}), these bounds reduce to
\begin{align*}
\Var((Q_h^U)^k f) &\leq A^k \Var(f) + (1+A+A^2+\dots+A^{k-1})B \norm{f}_{L^1}\\ &\leq A^k \Var(f) + \frac{B}{1-A} \norm{f}_{L^1},
\end{align*}
which is a classical iterated form of the Lasota-Yorke inequality~\cite{GaNi, KeLi}.
\end{remark}
For a general projection, instead, $E\neq 0$ and the matrix
\[
\mathcal{A}_h =
\begin{bmatrix}
    1 & 0\\
    E h & 1
\end{bmatrix}
\begin{bmatrix}
    A & B\\
    0 & 1
\end{bmatrix}
\]
has an eigenvalue strictly larger than $1$, hence $R_{k,h,1}$ and $R_{k,h,2}$ diverge and $\norm{Q_h^k}$ is not bounded uniformly in $k$.
Nevertheless, $\mathcal{A}_h$ is an $O(h)$ perturbation of the power-bounded matrix
$\mathcal{A}_0 = \begin{bsmallmatrix}
    A & B\\
    0 & 1
\end{bsmallmatrix}$, so these estimates can be shown to be useful when $k \ll 1/h$.

We can now prove a result that shows that discretizations of the same operator with different grid sizes are `close' (in a suitable sense). Let us consider two discretizations of the same Perron operator $L$, with $n$ and $n_F$ elements respectively (and grid sizes $h=1/n$, $h_F=1/n_F$) respectively. Note that if $n_F$ is a multiple of $n$, then for both $P_h^U$ and $P_h^L$ the finer grid is a refinement of the coarse grid, and $P_{h}P_{h_F} = P_{h_F}P_h = P_{h}$.

\begin{theorem}\label{thm:difference}
Let $Q_{h},Q_{h_F}$ be two ($i$-preserving) discretizations of the same Perron operator $L$, obtained with projections such that  $P_{h}P_{h_F} = P_{h_F}P_h = P_{h}$. Then, for each $f \in \mathcal{U}_{h_F}^0$ we have
\[
\norm{(Q_{h}^m  - Q_{h_F}^m)f} \leq 2 K h \sum_{k=0}^{m-1} C_{m-1-k}  \left(\norm{Q_{h_F}} \norm{Q_{h_F}^k f}_s + \norm{Q_{h_F}^{k+1}f}_s \right).
\]
where $\norm{Q^k_h|_{\mathcal{U}_h^0}}\leq C_k$.
\end{theorem}
\begin{proof}
The key insight is noticing that $L_{h_F} = P_{h_F} Q_{h} P_{h_F}$, so we can regard $Q_{h}$
as a further discretization of the operator $Q_{h_F}$, rather than a discretization of $L$.
In particular, we can apply Lemma~\ref{lemma:discbound} with $Q_{h_F}$ in place of $L$.
The rest follows once again from a telescopic sum argument.
\begin{align*}
\norm{(Q_{h}^m - Q_{h_F}^m)f} &\leq \sum_{k=0}^{m-1} \norm{Q_h^{m-1-k} (Q_h-Q_{h_F}) Q_{h_F}^{k}f}\\
& \leq \sum_{k=0}^{m-1} C_{m-1-k} 2 K h \left(\norm{Q_{h_F}} \norm{Q_{h_F}^k f}_s + \norm{Q_{h_F}^{k+1}f}_s \right).
\end{align*}
\end{proof}
\begin{corollary}\label{finalcor}
We have
\begin{equation} \label{finebounds}
\norm{Q_{h_F}^m|_{\mathcal{U}_{h_F}^0}} \leq C_m + 2K h \sum_{k=0}^{m-1} C_{m-1-k} (\norm{Q_{h_F}} R_{k,h_F,1} + R_{k+1,h_F,1}).
\end{equation}
\end{corollary}
\begin{proof}
From Theorem \ref{thm:difference} we have that for all $f\in \mathcal{U}^0_{h_F}$
\begin{align*}
\norm{Q_{h_F}^m f}&\leq \norm{Q_{h}^m f}+\norm{(Q_{h}^m - Q_{h_F}^m)f}\\
&\leq C_m \norm{f}+\sum_{k=0}^{m-1} C_{m-1-k} 2 K h \left(\norm{Q_{h_F}} \norm{Q_{h_F}^k f}_s + \norm{Q_{h_F}^{k+1}f}_s\right)
\end{align*}
Observe that by Corollary \ref{cor:Rkh} we have that 
\[
\norm{Q_{h_F}^k f}_s \leq R_{k,h_F,1} \norm{f}, 
\]
therefore
\[
\norm{Q_{h_F}^m f} \leq C_m \norm{f}+\sum_{k=0}^{m-1} C_{m-1-k} 2 K h \left(\norm{Q_{h_F}} R_{k,h_F,1} + R_{k+1,h_F,1}\right)\norm{f}.
\]

\end{proof}

This estimate requires only the explicit computation of $\norm{Q_{h_F}}$ and of the norms $C_k$ computed on a matrix of size $n < n_F$. Hence its computational cost is $O(n^2m + n_F)$, which can be much smaller than $O(n_F^2 m)$.

% \begin{remark}}
% It is worth remarking though, that this is not a uniform estimate working for all $h_F$. A notable example\todo{FP: example of what? Questo remarko va risistemato ora che ho tolto/spostato il primo pezzo in Lemma~\ref{lem:onembound}} is the Ulam approximation (Section \ref{sec:Ulam}) consisting of two copies of the doubling map $T(x) = 2x \mod 1$,
% one defined on $[0, \sqrt{2}/2]$ and the other 
% on $[\sqrt{2}/2, 1]$.

% While any Ulam discretization contracts the space of average $0$ functions, as the number of elements 
% in the discretization grows, the speed of contraction of the space of average $0$ measures slows down.
% \end{remark}

% \begin{remark}
% The estimate in \eqref{eq:differentdecayphenomena} is efficient in practice because there are two different decay phenomena that ensure that the summands are small: $C_{m-1-k}=O(\lambda_2^{m-1-k})$, which is small for small values of $k$, while (when $kh_F \ll n$) $R_{k,h_F,1} = O(\frac{1}{h_F}A^k)$, which is small for larger values of $k$. In particular, each term in the sum is $O(\frac{1}{h_F}\max(\lambda_2,A)^m)$.    
% \end{remark}

\begin{remark}
When used alone, this process to derive coefficients $C_m^F$ on a finer grid from coefficients $C_m$ on a coarser grid never gives a practical advantage when used in~\eqref{theorem1_estimate}. Indeed, ignoring some moderate factors and summands, we are replacing the estimate
\[
\norm{u-u_h} \sim h \sum_{k=0}^{\infty} C_k
\]
from Theorem~\ref{th:fixed_point} with
\begin{align*}
\norm{u-u_{h_F}} &\sim h_F \sum_{m=0}^{\infty} C_m^F \sim h_F \sum_{m=0}^{\infty} h \sum_{k=0}^{m-1} C_{m-1-k} \frac{1}{h_F}A^{k}\\
&=  h \sum_{m=0}^{\infty} \sum_{j=0}^{m-1} C_{j} A^{m-1-j} = h \sum_{j=0}^{\infty} \sum_{m = j+1}^\infty C_j A^{m-1-j} \\
&= \frac{h}{1-A} \sum_{j=0}^{\infty} C_j,
\end{align*}
from Theorem~\ref{thm:difference}; and this estimate is worse by a factor $\frac{1}{1-A}$. This rough computation suggests that one is always better off using the bound in Theorem~\ref{th:fixed_point} on $Q_h$ directly, forgoing $Q_{h_F}$ entirely.

However, another key ingredient is that we have other sources of \emph{a priori} bounds on $C_k^F$ which are more effective for small $k$ and improve this estimate significantly. These different bounds are described in detail in Section~\ref{sec:aggregating}.
\end{remark}

%A full algorithm that attempts to make use of all the available %information on matrix norms is present in Algorithm %\ref{algo:invariantmeasure}.

\section{Applying the general strategy to the transfer operators of nonsingular maps} \label{sec:transferoperator}
The main application of the  abstract approximation scheme we present is the approximation
of invariant densities for expanding and piecewise expanding dynamical systems on the unit interval $[0,1]$.

Let $T$ be a measurable map $T:[0,1] \to [0,1]$, we say $T$ is \textbf{nonsingular} if $m(T^{-1}(A))=0$ if and only if $m(A)$ is equal to $0$ for all measurable subsets $A$. 

Given a measurable map, the action of the dynamical
system extends to the space of probability measures through the push-forward operator associated to the map $T$, usually denoted as  $T^*$,
which associates to a probability measure $\mu$ the unique measure $T^*\mu$ such that
$(T^*\mu)(A) = \mu(T^{-1}A)$ for all measurable set $A$. 
If $T$ is nonsingular, the space of absolutely continuous measures is preserved by 
$T^*$; this induces an operator $L:L^1[0,1]\to L^1[0,1]$ on the space of densities,  called the \textbf{Perron-Frobenius operator} associated to the dynamical system.  It is well known that $L$ in this case is a weak contraction in $L^1$; for each $f\in L^1[0,1]$,  
\[\norm{Lf}_{L^1}\leq \norm{f}_{L^1}.\]

In the case where the map is piecewise expanding we have that the associated Perron-Frobenius operator satisfies a Lasota Yorke inequality. The following is a classical result,  see \cite{LY} or  \cite{GaNi}[Theorem 5.2] for a proof.

\begin{lemma}[$\Var-L^1$ Lasota-Yorke inequality]\label{thm:VarLY}
Let $T:[0,1]\to [0,1]$ and suppose there exists a finite partition
$\{P_k\}_{k=1}^b$ of $[0,1]$ such that
\begin{enumerate}
 \item $T_k = T|_{P_k}$ is $C^2$,
 \item $|T'(x)|>2$ for all $x\in [0,1]$
 \item the distortion $|T''(x)/T'(x)^2|$ is uniformly bounded by a constant $D$,
\end{enumerate}
then \eqref{genericLY} is satisfied with
\begin{equation}\label{eq:LY-Var}
A=\sup_x\left|\frac{2}{T'(x)}\right| \quad  B=\sup_k \frac{2}{|P_k|}+D,
\end{equation}

Mantaining hypothesis (1) and (3), relaxing hypothesis (2) to $|T'(x)|>1$ for all $x\in [0,1]$ and with
the addition that for all $k$ $f(P_k)=[0,1]$,
then \eqref{genericLY} is satisfied with
\begin{equation}\label{eq:LY-fullbranch}
A=\sup_x\left|\frac{1}{T'(x)}\right| \quad  B=D.
\end{equation}
\end{lemma}

In this context it is also well known (see e.g. \cite{V}) that the transfer operator associated to a piecewise expanding map $T$, provided that $T$ is  topologically mixing has a unique invariant probability density having bounded variation.

\subsection{Recalling the needed constants}
In the following we will use the basic facts recalled above for the approximation of invariant densities of examples of piecewise expanding maps. We will do this following our general strategy, for different discretizations and using different spaces. 
We recall that to apply our approximation strategy  we have to provide the following bounds:
\begin{itemize}
\item the coefficients $A,B$ of a Lasota-Yorke inequality
 \begin{equation*}
	\norm{Lf}_s \leq A \norm{f}_s + B\trinorm{f},  \quad A < 1,
\end{equation*}
\item the constant of the discretization error $K$,
\item the ``$i$-injection'' constant $E$,
\item the discretized ``strong-weak'' constant $M$,
\item the ``weak-strong+auxiliary'' constants $S_1$ and $S_2$,
\item a bound on $\norm{L}$.
\end{itemize}

In the next sections we will compute all these constants for the Ulam approximation and for the
piecewise linear approximation studied in \cite{GaNi} showing how the application of the coarse-fine  strategy brings a substantial improvement in the computing speed and in the precision.

\section{The Ulam projection}\label{sec:Ulam}
The first projection that we consider is the so-called Ulam projection on the torus.
Subdivide $[0,1)$ into $n$ intervals $I_j = [(j-1)h, jh)$, $j=1,\dots,n$, with the same width $h=1/n$, and define
\[
(P^{U}_h f) (x) = \frac{1}{h}\int_{I_j} f(y) dy, \quad \text{if $x\in I_j$}, \quad j=1,\dots,n.
\]
i.e., $P_h^{U} f$ is the piecewise constant function that is equal on each interval $I_j$ to the integral average of $f$ on $I_j$. Its image $\mathcal{U}_h$ is the space of piecewise constant functions on this grid. A natural basis for $\mathcal{U}_h$ is the one composed of the characteristic functions of the intervals $I_1, I_2,\dots, I_n$. In this basis, the coordinates of a function $f \in \mathcal{U}_h$ are $f_j = f((j-1)h)$ for $j=1,2,\dots,n$, and
\begin{equation} \label{L1l1}
    \norm{f}_{L^1} = \frac{1}{n} \sum_{j=1}^n \abs{f_i} = \frac{1}{n} \norm*{\begin{bmatrix}
        f_1\\
        f_2\\
        \vdots\\
        f_n
    \end{bmatrix}}_{\ell^1}.
\end{equation}
Moreover, the matrix associated to $L^U_h$ has elements
\begin{equation} \label{Pij_ulam}
	(L^U_h)_{ij} = \frac{\abs{T^{-1}(I_i) \cap I_j}}{\abs{I_j}}.
\end{equation}
This discretization admits a simple interpretation, first suggested by Ulam in \cite[pag.73-75]{Ulam}\footnote{ For the interested reader, it can be found at \url{https://archive.org/details/collectionofmath0000ulam/page/73}}:
$(L^U_h)_{ij}$ is the probability that a random point in $I_j$ (under the scaled Lebesgue measure)
is mapped by $T$ into the interval $I_i$. Hence $L^U_h$ is the transition matrix of a Markov chain which approximates
(in a suitable sense) the dynamic of the map $T$.

% The computation follows the following abstract scheme: for each $I_i$, we compute an object in the dual of $\mathcal{U}$, $D_i$ such that
% \[
% D_i(f) = \int \chi_{I_i}\circ T(x) f(x) dx.
% \]
% This dual object consists of the preimages of the endpoints of $I_i$ through $T$ and we can check which elements of the basis
% $\chi_{I_j}/|I_j|$ do not lie in its kernel by comparing this endpoints. This allows us to only compute the necessary coefficients in the matrix.

\begin{remark}
When discretizing the transfer operator of a piecewise expanding map, the matrix $L_h^U$ we obtain, with elements in~\eqref{Pij_ulam} is sparse. Indeed, we can decompose
\begin{equation} \label{Pij_ulamsum}
  (L^U_h)_{ij} = \sum_{i=1}^k \frac{\abs{T_k^{-1}(I_i) \cap I_j}}{\abs{I_j}},
\end{equation}
and by Lagrange's theorem, 
\[
\abs*{T_k^{-1}(I_i)} \leq \frac{h}{\inf \abs{T'}} < \frac{h}{2},
\]
hence $L_h^U$ has at most $2m$ nonzero elements in each row.
\end{remark}

In this section we will find all the needed constants for the Ulam projection; in the Ulam case, we use the following norms.
\begin{normsulam}
The norms involved in the Ulam approximation scheme are
\begin{itemize}
 \item the strong seminorm is $\norm{.}_s:=\Var{(.)}$,
 \item the weak norm is $\norm{.}:=\norm{.}_{L^1}$,
 \item $i(f)=\int f dm$, where $m$ is the Lebesgue measure on $[0,1]$.
\end{itemize}
The function $i(f)$ is represented in the above basis by the row vector $i^* = \frac1n [1,1,\dots,1]$. 
\end{normsulam}

\begin{comment}
This inequality is a classical result in the topic, we cite the following Local Variation Inequality as a reference.
\begin{lemma}[Lemma 56 \cite{GMN1}]
Let $I\subset [0,1]$ be an interval and let $L_i g$ be the component of $Lg$ coming from the $i$-th branch.
We have that
\[
\Var_I(Lg)\leq \sum \Var_I(L_i g)
\]
and
\begin{align*}
\Var_I(L_i g)&\leq \Var_{T^{-1}_i I}(g) \left |\left|\frac{1}{T'}\right|\right|_{L^{\infty}(T^{-1}_i(I))}+\left |\left|g\right|\right|_{L^{1}(T^{-1}_i(I))}\left |\left|\frac{T''}{(T')^2}\right|\right|_{L^{\infty}(T^{-1}_i(I))}\\
&+\sum_{y \in \partial \textrm{Dom}(T_i), T(y)\in I} \left|\frac{g(y)}{T'(y)}\right|.
\end{align*}
\end{lemma}
When dealing with full branch maps, the boundary term in the Local Variation Inequality is zero; this is the reason why the constants in \eqref{eq:LY-fullbranch} are different.
\end{comment}

\subsection{Establishing the necessary bounds}
In this subsection we estimate the necessary constants for our approximation procedure.
Most of the estimates are trivial or well known, and are proved for a matter of completeness.
\begin{lemma}
Let $P^U_h$ be the Ulam discretization on $n$-elements.
Then:
\begin{enumerate}
 \item $\norm{P_h f-f}\leq \frac{h}{2}\norm{f}_s$, therefore $K=1/2$,
 \item $\trinorm{P_h f}=\trinorm{f}$ and $\int_X (f-P_h f)dx=0$, therefore $E=0$,
 \item if $f_h\in \mathcal{U}_h$ we have that $\norm{f_h}_s\leq 2\frac{\norm{f_h}}{h}$, therefore $M=2$,
 \item $\norm{f}= \trinorm{f}$ therefore $S_1=0, S_2=1$.
\end{enumerate}
\end{lemma}
\begin{proof}
We refer to \cite{GMN1} for a proof of (1).
Since $P_h$ is a positive operator, we have that $\trinorm{P_h f}\leq \trinorm{P_h 1}=\trinorm{1}=1$;
moreover
\[
\int_0^1 f-P_h f dx = \sum_{i=0}^n \int_{I_i} \left(f(x)-\frac{1}{h}\int_{I_i} f(y)dy\cdot \chi_{I_i}(x)\right)dx=0,
\]
therefore $E=0$, item (2).

If $f_h\in \mathcal{U}_h$, we have that $f_h=\sum_{i=0}^n f_i \chi_{I_i}$ and
\[
\Var{f_h}=\sum_{i=0}^{n-1}\left|f_{i+1}-f_i\right|\leq 2\sum_{i=0}^m |f_i|\leq 2 h \norm{f}_{L^1},
\]
therefore $M=2$, item (3).

Item (4) follows from the fact that $\norm{f}=\norm{f}_{L^1}$.
\end{proof}

\subsection{Spectral picture for $L^U_h$}

Note that the Ulam projection is, by its definition, $i$-preserving, i.e., $i(P_h^Uf) = i(f)$. In particular, this implies that $L^U_h = Q^U_h$. 

We have $i^* Q_h^U = i^*$, hence $Q_h^U$ is a stochastic matrix, which is also irreducible and a-periodic by the mixing hypothesis.
By the Perron-Frobenius theorem, its largest eigenvalue is $\lambda_1=1$, and the associated eigenvector $u_h$
has strictly positive entries; moreover, the second largest eigenvalue is $\lambda_2 < 1$.
In particular, $1=\norm{Q_h^U}_{L^1} = \norm{(Q_h^U)^k}_{L^1}$ for all $k\in\mathbb{N}$,
while $\norm{(Q_h^U)^k|_{\mathcal{U}_h^0}} = O(\lambda_2^k)$, where
\begin{equation} \label{Uh0}
 	\mathcal{U}_h^0 := \{g \in \mathcal{U}_h : i^* g = 0\}.
\end{equation}

\section{The piecewise linear projection}\label{sec:PL}

In this section we will find all the needed constants for the piecewise linear projection on $[0,1)$.

The piecewise linear projection is defined as follows.
Divide $[0,1]$ into $n$ equal intervals, delimited by equispaced nodes $\{a_i = \frac{i}{n}\}_{i=0}^n$.
Let $\phi_i(x)$ be the piecewise linear function
\[
\phi_i(x)=\left\{\begin{array}{cc}
    n(x-a_{i-1}) & x\in [a_{i-1},a_i] \\
    -n(x-a_{i}) & x\in [a_{i},a_{i+1}]\\
    0 & x \in [a_{i-1},a_{i+1}]^c
        \end{array}\right.
\]
and define
\[
(P^{L}_h f) (x) = \sum_{i=0}^n f(a_i)\phi_i(x)
\]
i.e., $P_h^{L} f$ is the piecewise linear function that which interpolates $f(a_i)$ on the given nodes. The image $\mathcal{U}_h$ of $P_h^L$ is the space of piecewise linear functions on this grid. A natural basis for this space is $(\phi_i(x))_{i=1,\dots,n}$. Given a function $f \in \mathcal{U}_h$, its coordinates in this basis are $f_j= f(a_{j-1})$ for $j=1,2,\dots,n$, and
\[
\norm{f}_{L^\infty} = \max_{i=1,\dots,n} \abs{f_i} = \norm*{
    \begin{bmatrix}
        f_1\\
        f_2\\
        \vdots\\
        f_n
    \end{bmatrix}
}_{\ell^\infty}.
\]
The matrix associated to $L_h$ has elements
\begin{equation} \label{Pij_L}
	(L_h)_{ij} = \sum_{x\in T^{-1}(a_i)}\frac{\phi_j(x)}{|T'(x)|}.
\end{equation}

% As in the Ulam discretization the computation is treated through the definition of dual objects, i.e., for each point $a_i$ (and therefore, each row of the matrix), we compute an element $D_i$ of the dual of $\mathcal{U}$, such that
% \[
% D_i(f) = \sum_{y\in T^{-1}(a_i)}\frac{f(y)}{|T'(y)|}.
% \]
% This allows us to identify the coefficients we need to compute easily, due to the fact that each $\phi_j$ has support $[a_{j-1}, a_{j+1}]$.

\begin{normspiecewise}
The norms involved in the piecewise linear approximation scheme are
\begin{itemize}
 \item the strong norm $\norm{.}_s:=\norm{.}_{Lip}$,
 \item the weak norm $\norm{.}:=\norm{.}_{\infty}$,
 \item $i(f)=\int f dm$, where $m$ is the Lebesgue measure on $[0,1]$.
\end{itemize}
The function $i(f)$ is represented in the above basis by the row vector $i^* = \frac1n [1,1,\dots,1]$. 
\end{normspiecewise}

\subsection{Expanding maps and the Lasota Yorke inequality}
In this case we need to prove that the operator $L$ preserves a 
stronger norm; this is proved in the next theorem.

\begin{theorem}[$\Lip-L^1$ Lasota-Yorke inequality] \label{thm:LipLY}
Let $T$ be in $C^{2}(S^1)$, with $|T'(x)|>1$ and $|T''/(T')^2|<D$. Then,
an inequality \eqref{genericLY} holds with
\[
A = \sup_x \frac{(2D+1)}{|T'(x)|}\quad  B =D(D+1)
\]
\end{theorem}
\begin{proof}
Since $T\in C^{2}(S^1)$, $|T'(x)|>1$ there exists (at least)
one fixed point of $T$; we can label this fixed point as $0$
and see $T$ as a map satisfying \eqref{eq:LY-fullbranch};
in specific, there exists a partition $\{P_k\}_{k=1}^b$ such that
$T(P_k)=[0,1]$ for all $k$; we denote by $T_k:=T|_{P_k}$.
Please remark that $\lambda$ and $D$ are defined in the proof of
\eqref{eq:LY-fullbranch}.

\begin{align*}
&|Lf(x)-Lf(y)|=\bigg|\sum_k \frac{f(T_k^{-1}x)}{T'_k(T_k^{-1}x)}-\frac{f(T_k^{-1}y)}{T'_k(T_k^{-1}y)}\bigg|\\
&\leq \bigg|\sum_k \frac{f(T_k^{-1}x)-f(T_k^{-1}y)}{T'_k(T_k^{-1}x)}\bigg|+\bigg|\sum_k \frac{f(T_k^{-1}y)}{T'_k(T_k^{-1}y)}-\frac{f(T_k^{-1}y)}{T'_k(T_k^{-1}x)}\bigg|\\
&\leq \textrm{Lip(f)} \lambda |x-y| \sum_k \frac{1}{T'_k(T_k^{-1}x)}+\bigg|\sum_k \frac{f(T_k^{-1}y)}{T'_k(T_k^{-1}y)}\bigg(1-\frac{T'_k(T_k^{-1}y)}{T'_k(T_k^{-1}x)}\bigg)\bigg|.
\end{align*}
and
\begin{align*}
&\bigg|1-\frac{T'_k(T_k^{-1}y)}{T'_k(T_k^{-1}x)}\bigg|=\bigg|\frac{T'_k(T_k^{-1}x)-T'_k(T_k^{-1}y)}{T'_k(T_k^{-1}x)}\bigg|\\
&\leq \lambda \textrm{Lip}(T'_k)|T_k^{-1}x-T_k^{-1}y|\leq \lambda^2 \textrm{Lip}(T'_k) |x-y|.
\end{align*}
Hence
\begin{equation} \label{intermediateLY}
	\Lip(Lf) \leq \lambda \norm{L1}_\infty \Lip(f) + D \norm{Lf}_\infty.
\end{equation}
Now we use \eqref{eq:LY-fullbranch} to estimate
\[
\norm{Lf}_\infty \leq \Var(Lf) + \norm{Lf}_{L^1} \leq \lambda \Var(f) + D \norm{f}_{L^1} + \norm{f}_{L^1}\leq \lambda \Lip(f) + (D+1)\norm{f}_{L^1},
\]
and, in particular, for the constant function $1$, $\norm{L1}_\infty \leq D+1$. Plugging these two bounds into~\eqref{intermediateLY} we get the thesis.
\end{proof}

\begin{remark}
Under the same hypotheses, the matrix $L_h$ with elements in~\eqref{Pij_L} is sparse. Indeed, for each $x\in [0,1)$ at most two of the functions $\phi_j(x)$ are nonzero, hence $L_h$ has at most $2m$ nonzero elements in each row.
\end{remark}

\begin{remark}
We need $(2D+1)\lambda < 1$ for this to be a valid Lasota-Yorke inequality. If this property does not hold, then we can replace $T$ with one of its iterates $T^{k}$. Clearly $T^{k}$ has the same invariant measure as $T$; moreover, the values of $\lambda$ and $D$ are replaced by $\lambda^k$ and $D(1+\lambda+ \dots + \lambda^{k-1}) < \frac{D}{1-\lambda}$. In particular, for sufficiently large $k$ one has
$\lambda^k(2\frac{D}{1-\lambda}+1) < 1$, hence this strategy works. Note, though, that $T^k$ has $kb$ monotonic branches instead of $b$, hence the associated matrix is less sparse and the whole method is more computationally expensive.
\end{remark}

\subsection{Establishing the necessary bounds}
\begin{lemma}
Let $P^L_h$ be the piecewise linear discretization on $n$-elements.
Then:
\begin{enumerate}
 \item $\norm{P_h f-f}\leq \frac{h}{2}\norm{f}_s$, therefore $K=1/2$,
 \item $\trinorm{P_h f-1\cdot \int_X (f-P_h f)dx}\leq \trinorm{f}+\Lip(f)h/2$, therefore $E=1/2$,
 \item if $f_h\in \mathcal{U}_h$ we have that $\norm{f_h}_s\leq 2\frac{\norm{f}}{h}$, therefore $M=2$,
 \item $\norm{f}\leq \trinorm{f}+\norm{f}_s$ therefore $S_1=1, S_2=1$,
 \item $\norm{L} \leq D+1$.
\end{enumerate}
\end{lemma}
\begin{proof}
To prove item (1) we study $f|_{[a_i,a_{i+1}]}$ and suppose that $f(a_i)=0, f(a_{i+1})=0$;
in this case, $||f|_{[a_i,a_{i+1}]}||_{\infty}\leq \Lip(f)\cdot (x-a_i)$ and
$||f|_{[a_i,a_{i+1}]}||_{\infty}\leq \Lip(f)\cdot (a_{i+1}-x)$. This means that
the graph of $|f(x)|$ lies under the graphs of the linear functions $y_1(x)=\Lip(f)\cdot (a_{i+1}-x)$
and $y_2(x)=\Lip(f)\cdot (x-a_i)$ which intersect in $(a_{i+1}-a_i)/2$, therefore
\[\norm*{f|_{[a_i,a_{i+1}]}}_{L^\infty}\leq \Lip(f)\frac{h}{2} \]
and, since it is true for all $i$ we have
\[
\norm*{f-P^L_h f}_{L^\infty}\leq \Lip(f)\frac{h}{2}.
\]

This implies that
\[
\abs*{\int_0^1 f- P^L_h f dx} \leq \int_0^1 |f- P^L_h f|dx\leq  \Lip(f)\frac{h}{2}.
\]

From this follows
\[
\trinorm{P_h f-1\cdot \int_X (f-P_h f)dx}\leq\trinorm{P_h f}+\trinorm{1\cdot \int_X (f-P_h f)dx}\leq \trinorm{f}+\Lip(f)\frac{h}{2}.
\]
and that
\begin{align*}
\trinorm{P^L_h f}&= \int_0^1 |P^L_h f(x)|-|f(x)|+|f(x)| dx\leq \int_0^1|f(x)|+ ||P^L_h f(x)|-|f(x)|| dx\\
&\leq \trinorm{f}+\int_0^1 |P^L_h f(x)-f(x)|dx\leq \trinorm{f}+\Lip(f)\frac{h}{2},
\end{align*}
therefore $E=1/2$.

We prove now Item~(3). If $f_h\in \mathcal{U}_h$ we have that $f_h(x)=\sum_{i=0}^n a_i\cdot \phi_i(x)$,
where the $\phi(x)$ are piecewise linear.
Therefore
\[
\Lip(f)=\max_i\frac{|a_{i+1}-a_i|}{h}\leq 2 \frac{\max_i |a_i|}{h}=2\frac{\norm{f}}{h}
\]

Item~(4) follows from the fact that for $x\neq \tilde{x}$
\[
|f(\tilde{x})|=|f(x)+\frac{f(\tilde{x})-f(x)}{\tilde{x}-x}(\tilde{x}-x)|\leq |f(x)|+\Lip(f)\cdot |x-\tilde{x}|.
\]
Suppose now that $|f|$ attains its maximum in $\tilde{x}$, and integrate:
\[
\int_0^1 ||f||_{\infty} dx\leq \int_0^1  |f(x)|+\Lip(f)\cdot |x-\tilde{x}| dx \leq  ||f||_{L^1}+\Lip(f).
\]
Item~(5) follows from the fact that $\norm{L}_{L^\infty}=\norm{L1}_{L^\infty}$, since $L$ is a positive operator, and $\norm{L1}_{L^\infty} \leq D+1$ as in the proof of Theorem~\ref{thm:LipLY}.
\end{proof}

\subsection{Spectral picture for $L^L_h$ and $Q^L_h$}
Compared with the Ulam projection, the spectral picture is more blurry for $L^L_h$ and $Q^L_h$.
The matrix $L_h^L$ is still a non-negative matrix, but since $P_h^L$ is not $i$-preserving its first eigenvector
$\lambda_1$ is not in general equal to $1$.

The row vector $i^*$ is a left eigenvector of $Q_h^L$ with eigenvalue equal to $1$,
however, $Q_h^L$ is not a non-negative matrix, so we do not have all the results implied
by the Perron-Frobenius theory of Markov chains; in particular, $\trinorm{Q_h^L} > 1$ in general
(and our experimental results suggest that even the limit $\lim_{h\to 0} \trinorm{Q_h^L} = 1$ does \emph{not} hold).

Nevertheless, the results by Keller and Liverani (see Corollary~\ref{cor:KeLi}) ensure that $\lambda_2$
is smaller than $1$ for sufficiently small values of $h$.

\section{Practical computation}\label{sec:8}
In this section we present the results that permit us to efficiently compute the
objects and the constants involved in our treatment. There are three main points in the algorithm:
\begin{itemize}
  \item Computing an interval matrix $\mathbf{L} \ni L_h$ that encloses $L_h$;
  \item Computing a fixed point vector for $\mathbf{L}$;
  \item Computing norm estimates $C_k \geq \norm{Q_h^k|_{\mathcal{U}_h^0}}$ for $k=0,1,2,\dots,m$, and reaching a $m$ such that $C_m < 1$.
\end{itemize}
We will address them one by one in the next sections.

\subsection{Assembling the sparse matrices}
Recall that the projection $P_h$ permits us to build a discretization of $L$, the projected operator
$L_h := P_h L P_h$.

Then, $L_h: \mathcal{U}_h \to \mathcal{U}_h$ can be represented by a square matrix in a suitable
basis of $\mathcal{U}_h$.

We describe here a strategy to compute the matrix associated to $L_h$ for the case of the Ulam and piecewise linear projections on the torus $[0,1)$. 
With some abuse of notation, we will denote with the same symbol both the operator
(acting on functions on $[0,1]$) and the matrix that represents it. We assume that the dynamic $T$ is composed of $b$ continuous and monotonic branches $T_1,T_2,\dots,T_{b}$, whose domains form a partition of $[0,1)$.

The partition underlying the projection (which is typically equispaced) can be described by an increasing sequence $\mathcal{Y}: 0=y_0 < y_1 < y_2 < \dots < y_{n-1} < y_n=1$ that partitions $[0,1)$ of $T$ into $\bigcup_{j=i}^n I_j$, with $I_j = [y_{j-1},y_j)$. We assume that the co-domain $[0,1)$ of $T:[0,1) \to [0,1)$ is partitioned according to this sequence $\mathcal{Y}$; then, its domain $[0,1)$ is decomposed into $nb$ intervals $T_k^{-1}(I_j)$, some of them possibly empty; their endpoints are an increasing sequence $\mathcal{X}: 0 = x_0 < x_1 < \dots < x_N=1$ that defines a partition of the domain $[0,1)$ of $T$. We say that this sequence $\mathcal{X}$ is the \emph{pull-back} of the sequence $\mathcal{Y}$, and we denote it by $\mathcal{X} = T^{-1}(\mathcal{Y})$. An example is shown in Figure~\ref{fig:pullback}.
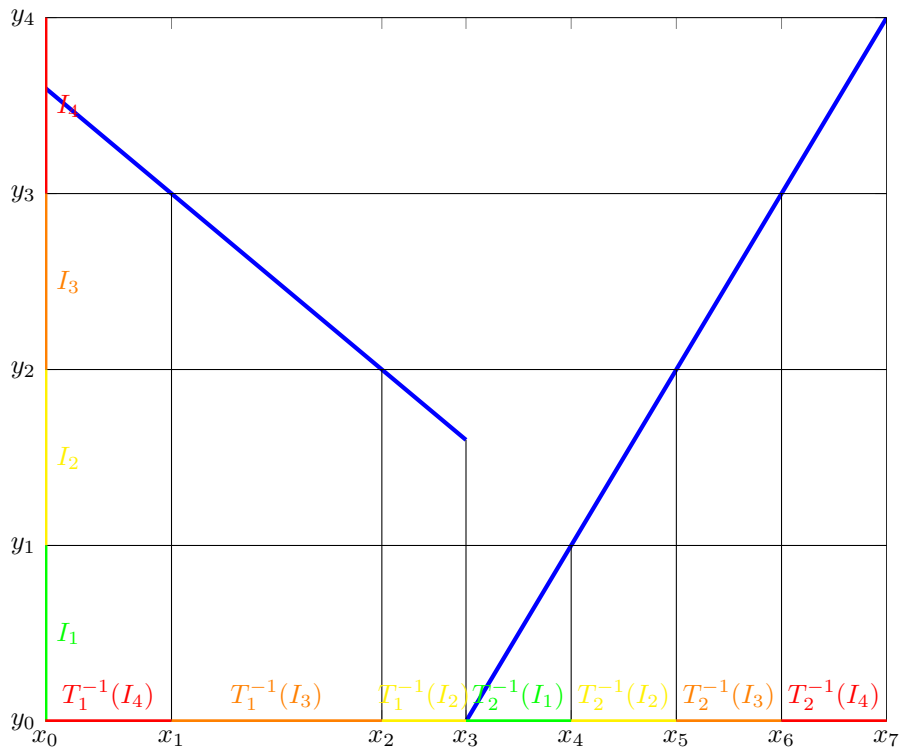
\begin{figure}[ht]
\begin{tikzpicture}
\begin{axis}[xmin=0, xmax=1, ymin=0, ymax=1, 
             axis line style={draw=none}, 
             width=\textwidth,
             ytick = {0, 0.25, 0.5, 0.75, 1},
             yticklabels = {$y_0$, $y_1$, $y_2$, $y_3$, $y_4$},
             xtick = {0, 0.15, 0.4, 0.5, 0.625, 0.75, 0.875, 1},
             xticklabels = {$x_0$, $x_1$, $x_2$, $x_3$, $x_4$, $x_5$, $x_6$, $x_7$}
            ]
\addplot[domain=0:0.5, draw=blue, ultra thick]{0.9-x};
\addplot[domain=0.5:1, draw=blue, ultra thick]{2*x-1};
\draw (axis cs:0, 0.25) -- (axis cs:1, 0.25);
\draw (axis cs:0, 0.5) -- (axis cs:1, 0.5);
\draw (axis cs:0, 0.75) -- (axis cs:1, 0.75);
\draw (axis cs:0, 1) -- (axis cs:1, 1);
\draw (axis cs:0.5,0) -- (axis cs:0.5,0.4);
\draw (axis cs:0.15,0.75) -- (axis cs:0.15,0);
\draw (axis cs:0.4,0.5) -- (axis cs:0.4,0);
\draw (axis cs:0.625,0.25) -- (axis cs:0.625,0);
\draw (axis cs:0.75,0.5) -- (axis cs:0.75,0);
\draw (axis cs:0.875,0.75) -- (axis cs:0.875,0);
\draw (axis cs:1, 1) -- (axis cs:1, 0);
\draw[red, ultra thick] (axis cs:0,0.75) -- node[right] {$I_4$}  (axis cs:0,1) ;
\draw[orange, ultra thick] (axis cs:0,0.5) -- node[right] {$I_3$} (axis cs:0,0.75);
\draw[yellow, ultra thick] (axis cs:0,0.25) -- node[right] {$I_2$} (axis cs:0,0.5);
\draw[green, ultra thick] (axis cs:0,0) -- node[right] {$I_1$} (axis cs:0,0.25);

\draw[red, ultra thick] (axis cs:0,0) -- node[above] {$T_1^{-1}(I_4)$} (axis cs:0.15,0);
\draw[orange, ultra thick] (axis cs:0.15,0) -- node[above] {$T_1^{-1}(I_3)$} (axis cs:0.4,0);
\draw[yellow, ultra thick] (axis cs:0.4,0) -- node[above] {$T_1^{-1}(I_2)$} (axis cs:0.5,0);
\draw[green, ultra thick] (axis cs:0.5,0) -- node[above] {$T_2^{-1}(I_1)$} (axis cs:0.625,0);
\draw[yellow, ultra thick] (axis cs:0.625,0) -- node[above] {$T_2^{-1}(I_2)$} (axis cs:0.75,0);
\draw[orange, ultra thick] (axis cs:0.75,0) -- node[above] {$T_2^{-1}(I_3)$} (axis cs:0.875,0);
\draw[red, ultra thick] (axis cs:0.875,0) -- node[above] {$T_2^{-1}(I_4)$} (axis cs:1,0);
\end{axis}
\end{tikzpicture}
\caption{An example of ``pull-back'': the pull-back along the map $T$ (drawn in blue) of the sequence $(y_j)$ on the y-axis is the sequence $(x_i)$ on the x-axis.} \label{fig:pullback}
\end{figure}  
The endpoints $x_i$ are either preimages $T_k^{-1}(y_j)$ for some $k$ and $j$, or endpoints of the domain of each branch; clearly we have $N \leq nb$, but some intervals may be missing if the map is not full-branch; for instance, in the example in Figure~\ref{fig:pullback} the interval $T^{-1}(I_1)$ is empty, and hence there are $7$ intervals instead of $8=4\cdot 2$ in the partition $\mathcal{X}$. 

Interval arithmetic methods such as the interval Newton method~\cite{Tucker} can be used to compute tight inclusion intervals $\mathbf{x}_i$ for each element $x_i$ of the pull-back partition, given explicit formulas to compute each branch of the map $T_k$. Once the $\mathbf{x}_i$ are available, inclusions $\mathbf{L}_{ij}$ for the matrix elements in either~\eqref{Pij_ulam} or~\eqref{Pij_L} are easy to compute.

The computation of the $\mathbf{x}_i$ can be performed automatically; we sketch how the method works for the second branch $T_2$ of the dynamic in Figure~\ref{fig:pullback}. One starts from the endpoints $(a,b)$ of $\operatorname{dom}(T_2)$. By checking how $T_2(a), T_2(b)$ compare with the elements of the sequence $\mathcal{Y}$, one can determine that $\mathbf{x}_3=a$, $\mathbf{x}_7=b$, and that three unknown values $x_4=T_2^{-1}(y_1), x_5=T_2^{-1}(y_2), x_6=T_2^{-1}(y_3)$ need to be computed. We can use a bisection strategy to reduce the number of iterations needed in the interval Newton method, as follows. We first compute $\mathbf{x}_5$ by applying the interval Newton method to find a zero of $x \mapsto T_2(x) - y_2$, using the whole domain  $\operatorname{hull}(\mathbf{x}_3,\mathbf{x}_7)$ as a starting interval. Once $\mathbf{x}_5$ has been computed, we obtain $\mathbf{x}_4$ by applying the interval Newton method to find a zero of $x \mapsto T_2(x) - y_1$ using the tighter interval $\operatorname{hull}(\mathbf{x}_3,\mathbf{x}_5)$ as a starting point instead of the whole domain, and similarly we use $\operatorname{hull}(\mathbf{x}_5,\mathbf{x}_7)$ as a starting interval in the interval Newton method to compute $\mathbf{x}_6$.

\begin{remark}
Since each branch of $T$ is expanding, the preimage problem is well-conditioned, and we expect to be able to compute enclosures with radius $\operatorname{rad}(\mathbf{x}_i)$ of the same order of magnitude as the machine precision used.
\end{remark}

\begin{remark}
This description in terms of pull-backs of partitions has the additional benefit that pull-backs of composed maps are particularly easy to compute, since $(S\circ T)^{-1}(\mathcal{Y}) = T^{-1}(S^{-1})(\mathcal{Y}))$.
\end{remark}

An explicit algorithm to compute a sparse interval matrix $\mathbf{L}\ni L_h$ is sketched in Algorithm~\ref{algo:assembly}. It has complexity $O(nb)$, since the sets $S_\ell$ have dimension $O(1)$. The algorithm returns the sparse matrix in coordinate list format, i.e., a list $\mathcal{S}$ of triples $(i,j,\mathbf{c})$ such that $\mathbf{L}_{ij} = \sum_{(i,j,\mathbf{c})\in\mathcal{S}} \mathbf{c}$. Note that the list $\mathcal{S}$ will in general contain multiple entries with the same $i$ and $j$.
\begin{algorithm}
\caption{Assembling a sparse interval matrix $\mathbf{L} \ni L_h$ (for the Ulam or piecewise linear discretization)} \label{algo:assembly}
\begin{algorithmic}[1]
\Function{assemble\_Lh}{$T,n$}
\Require{$T$, partition $\mathcal{Y}$ (typically an equispaced one)}
\Ensure{A list $\mathcal{S}$ of triples $(i,j,\mathbf{c})$}
\State{compute enclosures $(\mathbf{x}_\ell)_{\ell=1}^N$ for the pull-back $\mathcal{X} = T^{-1}(\mathcal{Y})$,}
\For{$\ell=1,2,\dots,N$}
\State determine the set $S_\ell = \{j:(\mathbf{x}_{\ell-1},\mathbf{x}_\ell)\cap I_j \neq \emptyset\}$ or $S_\ell = \{j:\phi_j(\mathbf{x}_\ell)\neq 0\}$ via a binary search on $\mathcal{Y}$;
\For{all $j\in S_\ell$}
\State push $(i,j,\mathbf{c})$ into $\mathcal{S}$, where $\mathbf{c}$ is a summand of~\eqref{Pij_ulamsum} or~\eqref{Pij_L},
\State and $i$ is the index such that $T((x_{\ell-1},x_{\ell})) \subseteq (y_{i-1},y_i)$; 
\EndFor
\EndFor
\EndFunction
\end{algorithmic}
\end{algorithm}

\subsection{Numerically approximating the fixed point}
We compute numerically an approximate fixed point $\tilde{u}_h$ of the operator $Q_h$ by using the restarted Arnoldi method~\cite[Section~10.5]{GVL} to return its eigenvector with eigenvalue (approximately) 1. While $L_h$ is a sparse matrix, $Q_h$ is not, in general. However, we can compute its (approximate) action on a vector $v$ using
\[
Q_h v \approx \operatorname{mid}(\mathbf{L}) v + e i^* (v - \operatorname{mid}(\mathbf{L}) v).
\]
For the Ulam discretization, $Q_h=L_h$ and we can drop the second summand.

In general, the computed eigenvector will not satisfy the equality $i^*\tilde{u}_h=1$ exactly. The following corollary of Theorem
\ref{th:fixed_point} allows us to estimate the distance between the computed $\tilde{u}_h$ and the exact fixed point of the operator.
\begin{corollary} \label{cor:fixedpoint2}
Under the hypothesis of Theorem~\ref{th:fixed_point} and Lemma~\ref{lem:onembound},
let $\tilde{u}_h$ be a vector such that $\norm{Q_h \tilde{u}_h - \tilde{u}_h} \leq \varepsilon_1$ and $|i^*\tilde{u}_h-1| \leq \varepsilon_2 < 1$.
Then,
\begin{equation*}
\norm{u-\tilde{u}_h} \leq \frac{C_0+C_1+\dots+C_{m-1}}{1-C_m}(2 K h \left(1+\norm{L}\right)\norm{u}_s + \frac{\varepsilon_1}{1-\varepsilon_2})+\frac{\varepsilon_2}{1-\varepsilon_2}\norm{\tilde{u}_h}.
\end{equation*}
\end{corollary}
\begin{proof}
Set $u_h=\tilde{u}_h/i^*\tilde{u}_h$. Then,
\begin{align*}
\norm{Q_h u_h - u_h} &= \frac{\norm{Q_h \tilde{u}_h - \tilde{u}_h}}{i^*\tilde{u}_h} \leq \frac{\varepsilon_1}{1-\varepsilon_2}, \\
\norm{u_h - \tilde{u}_h} &\leq  \norm{\frac{1-i^*\tilde{u}_h}{i^*\tilde{u}_h} \tilde{u}_h} \leq \frac{\varepsilon_2}{1-\varepsilon_2} \norm{\tilde{u}_h}.
\end{align*}
Combining these two bounds with~\eqref{theorem1_estimate} and the first point of Lemma~\ref{lem:onembound} gives the desired result.
\end{proof}

\subsection{Bounding norms of powers computationally}
In this section, we describe a computational procedure to obtain rigorous bounds of the form $\norm{Q_h^k|_{\mathcal{U}^0_h}} \leq C_k$ in practice on a computer. We start by recalling one important notation convention we stated in Notation \ref{not:general}
\begin{notation}
The symbol $\norm{f}_{L^p}$ denotes the $L^p$ norm of a function (usually defined on $[0,1]$), whereas the symbol $\norm{v}_{\ell^p}$ denotes the $\ell^p$ norm of a vector $v\in\mathbb{R}^n$.
\end{notation}

In the Ulam projection, since the `continuous' norms $\norm{\cdot}_{L^1}$ and the `discrete' norm $\norm{\cdot}_{\ell^1}$ differ only by a constant (see~\eqref{L1l1}), we have $\norm{Q_h^k|_{\mathcal{U}_h^0}}_{L^1} = \norm{Q_h^k|_{\mathcal{U}_h^0}}_{\ell^1}$, and similarly for $L^\infty$ and $\ell^{\infty}$ in the piecewise linear projection. Hence we can replace these norms with matrix norms for which there are classical formulas
\begin{equation} \label{eq:classicalnorms}
	\norm{M}_{\ell^1} = \max_{i} \sum_j \abs{M_{ij}}, \quad \norm{M}_{\ell^{\infty}} = \max_{j} \sum_i \abs{M_{ij}}.
\end{equation}
However, even after reducing to a discrete setting, computing matrix norms restricted to a certain subspace $\mathcal{U}_h^0$ is not a textbook problem. The following bound allows one to solve it.

% \begin{remark}
% Since in this section we are dealing with matrices we will use the usual $p$-norm on a vector space of dimension $n$:
% \[
% \norm{x}_p = (\sum_{i=0}^{n-1} x_i^p)^{1/p}.
% \]

% Let $f$ be a function in $\mathcal{U}_h$. Then $f(x)=\sum_{i=0}^{n-1} a_i\phi_i(x)$ where $\{\phi_i\}$ is a basis of $\mathcal{U}_h$.
% Then
% \[
% ||f||_p=\sqrt[p]{\int_0^1 |f|^p dm}\leq \sqrt[p]{\sum_{i=0}^{n-1}|a_i|^p \int_0^1 |\phi_i(x)|^p dm}.
% \]
% We can always choose the basis $\phi_i(x)$ in such a way that, if $x=(a_0, a_1, \ldots, a_{n-1})$:
% \[
% ||f||_p=C||x||_p,
% \]
% so, if $L_h:\mathcal{U}_h\to \mathcal{U}_h$ is a linear operator represented by a matrix $M$ in the basis $\{\phi_i\}$ we have that
% \[
% ||L_h||_p = ||M||_p,
% \]
% i.e. the corresponding operator norms coincide.
% \end{remark}

\begin{lemma} \label{matrixnormlemma}
Let
\[
U = \begin{bmatrix}
    1 & 1 &  \cdots & 1\\
    -1 & 0 & \cdots & 0\\
    0& -1 & \ddots & \vdots\\
    \vdots& \ddots& \ddots& 0\\
    0 & \cdots & 0 & -1
\end{bmatrix} = \begin{bmatrix}
    e^*\\
    -I
\end{bmatrix} \in \mathbb{R}^{n \times (n-1)},
\]
and $\mathcal{U}_h^0 = \ker([1,1,\dots,1]^*) = \operatorname{Im} U$.
Then, for each $M\in \mathbb{R}^{n\times n}$ and each $\ell^p$ norm one has $\norm{M|_{\mathcal{U}_h^0}} \leq \norm{MU}$.
\end{lemma}
\begin{proof}\relax
We have for each $z\in\mathbb{R}^{n-1}$
\[
\norm{Uz} = \norm*{
\begin{bmatrix}
    z_1 + z_2 + \dots + z_{n-1}\\
	-z_1\\
	-z_2\\
	\vdots\\
	-z_{n-1}
\end{bmatrix}
} \geq \norm{z}.
\]
Moreover,
\[
\norm{M|_{\mathcal{U}^0_h}} = \sup_{x \in \mathcal{U}_h^0 \setminus \{0\}} \frac{\norm{Mx}}{\norm{x}} = \sup_{z \in \mathbb{R}^{n-1} \setminus \{0\}} \frac{\norm{MUz}}{\norm{Uz}} \leq \sup_{z \in \mathbb{R}^{n-1} \setminus \{0\}} \frac{\norm{MUz}}{\norm{z}} = \norm{MU}.
\]
\end{proof}
%\begin{remark}
%\todo{FP: I think these two remarks are wrong, because we care about a bound to $\norm{Uz}/\norm{z}$ from \emph{below}, not \emph{above}. So what matters is \emph{not} the norm of $U$.}

%This is an upper bound, since the norm $||U||$ is bigger than $1$ in general.
%For the $\norm{.}_1$:
%\[
%\norm{U e_1}_1 = \norm{e_1-e_2}_1=2,
%\]
%while for $\norm{.}_{\infty}$ we have
%\[
%\norm{U e}_{\infty} = n.
%\]
%\end{remark}

\begin{remark}
Note that $\norm{Uz}_{\ell^1} \leq 2\norm{z}_{\ell^1}$ and $\norm{Uz}_{\ell^\infty} \leq (n-1)\norm{z}_{\ell^\infty}$, so this bound is off by at most a factor $2$ in the $\ell^1$ norm and by at most a factor $n-1$ in the $\ell^\infty$ norm.
\end{remark}

\begin{remark}
For a generic projection, an analogous procedure can be devised. Let $U\in\mathbb{R}^{n\times (n-1)}$ be a matrix whose columns are a basis of $\mathcal{U}_h^0=\ker i^*$, and suppose that
$||Uz|| \geq \alpha||z||$.
Then, by the same reasoning, we have that
\[
\norm{M|_{\mathcal{U}^0_h}} = \sup_{x \in \mathcal{U}_h^0 \setminus \{0\}} \frac{\norm{Mx}}{\norm{x}} = \sup_{z \in \mathbb{R}^{n-1} \setminus \{0\}} \frac{\norm{MUz}}{\norm{Uz}} \leq \sup_{z \in \mathbb{R}^{n-1} \setminus \{0\}} \frac{\norm{MUz}}{\alpha \norm{z}} = \frac{1}{\alpha} \norm{MU}.
\]

An estimate for $\alpha$ can be obtained automatically for any norm $||.||$ for which we know explicit constants $c,C$ such that
\[
c||z||_{\ell^2}\leq ||z||\leq C||z||_{\ell^2},
\]
using a rigorous estimate for
$$\sigma_{\min} = \min(||Uv||_{\ell^2}/||v||_{\ell^2})$$
obtained from the SVD decomposition of $U$, using techniques to rigorously
certify eigenvalues as in \cite{Miyajima}. 

Therefore
\[
||Uz||\geq c||Uz||_{\ell^2}\geq c\eta ||z||_{\ell^2}\geq c\sigma_{\min} C ||z||.
\]
\end{remark}

\begin{remark}
In the case of a more general weak norm $\norm{\cdot}$, we can reduce the problem to the computation
of the $\ell^1$ and $\ell^{\infty}$ norms of the operator.
To do so, we need three estimates
\[
\norm{v}\leq W_1\norm{v}_{\ell^1}+W_2\norm{v}_{\ell^\infty}
\]
and
\[
\norm{v}_{\ell^1}\leq \alpha_1\norm{v}, \quad \norm{v}_{\ell^\infty}\leq \alpha_{\infty} \norm{v},
\]
which imply
\[
\norm{P}\leq \frac{W_1}{\alpha_1}\norm{P}_{\ell^1}+\frac{W_2}{\alpha_2}\norm{P}_{\ell^\infty}.
\]
\end{remark}

\begin{remark}
There is some linear algebra literature on fast estimation of matrix norms, for instance~\cite{higham-norms}, but unfortunately we cannot use it here. Indeed, these estimators return only a guaranteed \emph{lower} bound $C \leq \norm{M}$. Providing a lower bound is a simpler problem, since it is sufficient to show that $\norm{Mx}\geq C$ for a suitable norm-1 vector $x$; giving a rigorous upper bound, instead, requires proving that $\norm{Mx} \leq C$ for \emph{all} norm-1 vectors.
\end{remark}

\subsection{Handling machine arithmetic errors when bounding norms}\label{subsec:boundmachine}

In principle, one can obtain a rigorous estimate for $\norm{Q_h|_{\mathcal{U}_h^0}}$ from the results in the previous section by computing $\norm{\mathbf{L}U}$ using interval arithmetic; however, matrix-vector products in interval arithmetic may be slow (as was the case for our computational environment), so we describe here an alternative procedure in which the matrix-vector products are computed using floating-point arithmetic: we replace $L_h$ with the floating-point matrix $M = \operatorname{mid}(\mathbf{L})$, and keep track of the error directly, in a sort of normwise ball arithmetic, bounding the error with $\delta = \norm{\operatorname{rad}(\mathbf{L})}$. We work out the required bounds in this section, for both the $\ell^1$ norm (used in the Ulam projection) and the $\ell^\infty$ norm (used in the piecewise linear projection). We first need to bound the computational error produced by products with $M$.
\begin{lemma} \label{lem:inexactmatprod}
Given $M\in\mathbb{R}^{n\times n}$ and $v\in\mathbb{R}^{n}$, let $\tilde{w} = \mathsf{fl}(Mv)$ be the vector obtained by evaluating the product $w = Mv$ in an inexact floating-point arithmetic system with machine precision $\mathsf{u}$. Then, for both norms $\norm{\cdot}_{\ell^1}$ and $\norm{\cdot}_{\ell^\infty}$, it holds that
\[
\norm{\tilde{w} - w} \leq \gamma_z \norm{M}\norm{v},
\]
where $\gamma_z := \frac{z\mathsf{u}}{1-z\mathsf{u}}$, and $z$ is the maximum number of nonzero entries in a row of $M$.
\end{lemma}
\begin{proof}
This result follows from~\cite[Section~3.5]{higham}, after noting that for a sparse matrix we can replace $\gamma_n$ with $\gamma_z$, since each sum has at most $z$ terms (as already argued in~\cite{GaNi}).
\end{proof}

The main results used to bound the total error are the following. The simplest case is that of an $i$-preserving projection, for which $Q_h = L_h$.
\begin{lemma}
Let $\tilde{v}_0 = v_0\in \mathbb{R}^{n}$ be a given fixed vector, and let $M\in\mathbb{R}^{n\times n}$ be a matrix such that $\norm{L_h-M} \leq \delta$. For each $k=1,2,\dots$, let $\tilde{v}_{k+1} := \mathsf{fl}(M \tilde{v}_k)$ be the vector obtained by evaluating the product $M v_k$ in floating-point arithmetic, and let the sequence $\epsilon_k$ be defined recursively as
\begin{equation} \label{defepsilonk1}
    \epsilon_0 = 0, \quad \epsilon_{k+1} = \gamma_z \norm{M}\norm{\tilde{v}_k} + \delta \norm{\tilde{v}_k} + \norm{L_h} \epsilon_k.
\end{equation}
Then,
\[
\norm{\tilde{v}_{k} - (L_h)^{k} v_0}_{\ell^1} \leq \epsilon_{k}, \quad k=0,1,2,\dots.
\]
\end{lemma}
\begin{proof}
Arguing by induction, we have
\begin{align*}
\norm{\tilde{v}_{k+1}-L_h^{k+1}v_0} &\leq \norm{\tilde{v}_{k+1}- M\tilde{v}_k} + \norm{M\tilde{v}_k - L_h \tilde{v}_k} + \norm{L_h(\tilde{v}_k - L_h^k v_0)}\\
& \leq \gamma_z \norm{M}\norm{\tilde{v}_k} + \delta \norm{\tilde{v}_k} + \norm{L_h} \epsilon_k. \qedhere
\end{align*}
\end{proof}
Note that for the Ulam projection $\norm{L_h^U}_{\ell^1}=1$, so we can remove that factor.

If the projection is not $i$-preserving, the corresponding estimate for $Q_h$ is slightly more involved, because we have to keep track of the second summand in $Q_h = L_h  + ei^*(I - L_h)$. Let us introduce the matrix $N = I - ei^*$, so that
\[
Q_h v = NL_hv + ei^*v,
\]
and the second summand vanishes if $v\in\mathcal{U}_h^0$. This suggests that we can approximate the action of $Q_h$ with that of $NL_h$.
\begin{lemma}
Let $\tilde{v}_0 = v_0\in \mathcal{U}_h^0$ be a given fixed vector, and let $M\in\mathbb{R}^{n\times n}$ be a matrix such that $\norm{L_h-M} \leq \delta$. For each $k=1,2,\dots$, let $\tilde{w}_{k+1} = \mathsf{fl}(M\tilde{v}_k)$ and
\begin{equation} \label{vkp1}
    \tilde{v}_{k+1} = \mathsf{fl}(N\tilde{w}_{k+1}) = \mathsf{fl}\left(\tilde{w}_{k+1} - e\left(\sum_{i=1}^n f_i (\tilde{w}_{k+1})_i\right)\right)
\end{equation}
be the vectors obtained by approximating $Q_h \tilde{v}_k$ in floating-point arithmetic, and let the sequence $\epsilon_k$ be defined recursively as
\begin{align} \label{defepsilonkinf}
    \epsilon_0 &= 0,\nonumber\\ \quad \epsilon_{k+1} &=  \gamma_{n+2}\norm{\begin{bmatrix}
    I & -ei^*
\end{bmatrix}}\left(\norm{w_{k+1}}+ \norm{ei^*}\norm{\tilde{w}_k}\right)\nonumber\\ &+ \norm{N} (\gamma_z \norm{M} + \delta)\norm{\tilde{v}_k} + \norm{Q_h}\epsilon_k.
\end{align}
Then,
\[
\norm{\tilde{v}_{k} - Q_h^{k} v_0}_{\ell^\infty} \leq \epsilon_{k}, \quad k=0,1,2,\dots.
\]
\end{lemma}
\begin{proof}
Standard forward error analysis of the formula~\eqref{vkp1} gives
\[
\norm*{\tilde{v}_{k+1} - N\tilde{w}_{k+1}} \leq \norm{\begin{bmatrix}
    I & -ei^*
\end{bmatrix}}\gamma_{n+2} \norm{\tilde{w}_{k+1}}.
\]
This bound is essentially the same that would follow from applying Lemma~\ref{lem:inexactmatprod} to the product $\tilde{v}_{k+1} = \begin{bmatrix}
    I & -ei^*
\end{bmatrix} \tilde{w}_k$, only with $\gamma_{n+2}$ instead of $\gamma_{n+1}$ because forming the products in $e i^*$ could in principle introduce another relative error of the magnitude of the machine precision. (Note that this additional error term can be omitted in the case of the piecewise linear discretization, since $e$ is the vector of all ones and products with its entries are exact.)

Moreover,
\begin{align*}
\norm{(NL_h - Q_h)\tilde{v}_k} &= \norm{ei^*\tilde{v}_{k}}\\ &= \norm{ei^*(\tilde{v}_{k}-N\tilde{w}_k)} \\&\leq \norm{ei^*}\norm{\tilde{v}_{k} - N\tilde{w}_{k}} \\&\leq \norm{ei^*} \norm{\begin{bmatrix}
    I & -ei^*
\end{bmatrix}}\gamma_{n+2} \norm{\tilde{w}_{k}}.
\end{align*}
Once we have established these bounds, we can conclude with the triangle inequality:
\begin{align*}
\norm{\tilde{v}_{k+1}-Q_h^{k+1}v_0} & \leq \norm{\tilde{v}_{k+1}-N\tilde{w}_{k+1}} + \norm{N(\tilde{w}_{k+1}-M\tilde{v}_k)} + \norm{N(M-L_h)\tilde{v}_k} \\
& \quad + \norm{(NL_h - Q_h)\tilde{v}_k} + \norm{Q_h(\tilde{v}_k - Q_h^{k}v_0)}\\
& \leq \norm{\begin{bmatrix}
    I & -ei^*
\end{bmatrix}}\gamma_{n+2}\norm{w_{k+1}} + \norm{N} \gamma_z \norm{M}\norm{v_k} + \norm{N}\delta \norm{v_k} \\
& \quad  + \norm{ei^*} \norm{\begin{bmatrix}
    I & -ei^*
\end{bmatrix}}\gamma_{n+2}\norm{w_k} + \norm{Q_h}\epsilon_k.
\end{align*}
\end{proof}
In the case of the piecewise linear projection, $\norm{ei^*}_{\ell^\infty}=1$ and
$\norm{\begin{bmatrix}
    I & -ei^*
\end{bmatrix}}_{\ell^\infty} = \norm{N}_{\ell^\infty} = 2$.

All the norms appearing in these lemmas can be replaced with computable bounds from above. To obtain a bound for $Q_h$, we can use
\[
\norm{Q_h} = \norm*{M + (L_h-M) + e(i^*-i^*L_h)} \leq \norm{M} + \delta + \norm{e}\norm{i^* - i^*L_h}.
\]
A bound $\norm{i^* - i^*L_h} \leq \norm{i^* - i^*\mathbf{L}}$ can be computed with a single vector-matrix product performed in interval arithmetic. In practice this approach performed quite well in our examples, since $\norm{i^* - i^*\mathbf{L}}$ is quite small for all the experiments described in Section \ref{sec:9}.

A full algorithm, for both the $\ell^1$ and $\ell^\infty$ norms, is sketched in Algorithm~\ref{algo:powernorms}. If $M$ has at most $z$ nonzeros in each row, this computation requires $O(n^2zk_{\max})$ arithmetic operations.
\begin{algorithm}
\begin{algorithmic}[1]
\Function{norm1\_of\_powers}{$M$, $k_{\max}$}
\Require{$M=\operatorname{mid}(\mathbf{L})$, $\delta=\norm{\operatorname{rad}(\mathbf{L})}$}
\Ensure{bounds $C_k \geq \norm{(Q_h)^k |_{\mathcal{U}_h^0}}_{\ell^1}$ for $k=1,2,\dots,k_{\max}$}
\For{$k=1,\dots,k_{\max}$}
	\State $C_k \gets 1$
\EndFor
\For{$j=1,2,\dots,n-1$}
	\State $v \gets e_1 - e_{j+1}$;
	\For{$k=1,\dots,k_{\max}$}
		\State $v \gets Mv$ \label{row:error1}  \Comment{Rounding to nearest}
        \State $v \gets v - ei^*v$  \Comment{Skipped if $i$-preserving}
		\State $C_k \gets \max(C_k, \norm{v}_{\ell^1} + \epsilon_k)$ \label{row:error2} \Comment{Rounding up; $\epsilon_k$ as in~\eqref{defepsilonk1} or~\eqref{defepsilonkinf}}
	\EndFor
\EndFor
\EndFunction
\Function{norminf\_of\_powers}{$M$, $k_{\max}$}
\State\Comment{compute bounds $C_k \geq \norm{(Q_h)^k|_{\mathcal{U}_h^0}}_{\ell^\infty}$ for $k=1,2,\dots,k_{\max}$}
\For{$k=1,\dots,k_{\max}$, $i=1,\dots,n$}
	\State $S_{ik} \gets 0$
\EndFor
\For{$j=1,2,\dots,n-1$}
	\State $v \gets e_1 - e_{j+1}$;
	\For{$k=1,\dots,k_{\max}$}
		\State $v \gets Mv$ \label{row:error3} \Comment{Rounding to nearest}
        \State $v \gets v - ei^*v$  \Comment{Skipped if $i$-preserving}
		\State $S_{ik} \gets S_{ik} + \abs{v_i} + \epsilon_k$ \label{row:error4} \Comment{Rounding up; $\epsilon_k$ as in~\eqref{defepsilonk1} or~\eqref{defepsilonkinf}}
	\EndFor
\EndFor
\For{$k=1,\dots,m$}
	\State $C_k \gets \max_i S_{ik}$
\EndFor
\EndFunction
\end{algorithmic}
\caption{Algorithms to estimate norms of powers} \label{algo:powernorms}
\end{algorithm}

\begin{remark}
It follows from~\eqref{defepsilonkinf} that $\epsilon_{k+1} \geq \norm{Q_h}^k \epsilon_1$, i.e., in the non-$i$-preserving case the bounds grow by at least a factor $\norm{Q_h}$ at each iteration. A more careful analysis could be made to replace some terms $\norm{Q_h}^k$ with $\norm{Q_h^k}$; we have implemented that and combined it with the bounds~\eqref{RSbound}, but in the end we observed no practical advantage, since the bounds produced by~\eqref{RSbound} are much worse than $\norm{Q_h}^k$ for moderate values of $k$, see
Figures \ref{fig:normcompare1} and~\ref{fig:normcompare2}.
\end{remark}

\subsection{Aggregating norm bounds from various sources} \label{sec:aggregating}

Bounds on the form $\norm{Q_h^k|_{\mathcal{U}^0_h}} \leq C_k$ come from various sources, some \emph{a priori}, some requiring explicit computation:
\begin{enumerate}
  \item $\norm{Q_h^k|_{\mathcal{U}_h^0}} \leq \norm{Q_h^k} \leq \norm{Q_h}^k \leq (\norm{\mathbf{L}}+\norm{e}\norm{i^*-i^*\mathbf{L}})^k$, from basic norm properties. For the Ulam discretization, $\norm{Q_h}=1$, hence this bound is the constant 1. This norm is fast to compute, and effective for low values of $k$, but it will never get below $1$, as $\norm{Q_h}\geq 1$. 
  \item $\norm{Q_h^k|_{\mathcal{U}_h^0}} \leq \norm{Q_h^k} \leq S_1 R_{k,h,1} + S_2 R_{k,h,2}$, from~\eqref{RSbound}. For the Ulam discretization, $E=0$, hence this bound is once again the constant 1. This a-priori bound requires only the Lasota--Yorke inequality constants, but it is typically equal of worse than the other alternatives.
  \item $\norm{Q_h^k|_{\mathcal{U}_h^0}} \leq \min_{0< i< k}(C_i C_{k-i})$, which comes from the sub-multiplicativity of norms and the fact that $Q(\mathcal{U}_h^0) \subseteq \mathcal{U}_h^0$. This estimate is based on the bounds $C_i$ obtained for $i<k$ with the other methods, but once those are available it is cheap to compute and effective. It becomes useful only after bounds smaller than 1 have already been obtained for at least some $i<k$.
  \item[(4a)] computational estimates obtained with Algorithm~\ref{algo:powernorms}. These bounds can be poor for small values of $k$, but they are our only resource to get non-trivial bounds smaller than 1 in the first place. As their cost scales with $O(n^2)$, these can be computed effectively only for discretizations with moderate $n$.
  \item[(4b)] estimates obtained from a coarser grid using~\eqref{finebounds}. These bounds are an effective replacement of those in Item~4a when $n$ is large. Exactly like the bounds in Item~4a, these are typically poor for small values of $k$, but they are the key ingredient to achieve bounds smaller than 1 in the coarse-fine strategy.
\end{enumerate}
For each $k$, our upper bound $C_k$ is the minimum of the bounds coming from items (1)--(4a) (or (4b)). It is essential to use multiple sources of bounds: the bound in item (1) is effective for small values of $k$; the bound in item (4) is the only one that can go below 1, and the bound in (3) can be used to combine the other ones and extend them to larger values of $k$. We plot in Figures~\ref{fig:normcompare1} and~\ref{fig:normcompare2} the norm bounds obtained from all these sources, on two representative examples with the Ulam and piecewise linear discretizations.
\begin{figure}
\centering
\includegraphics[width=0.48\textwidth]{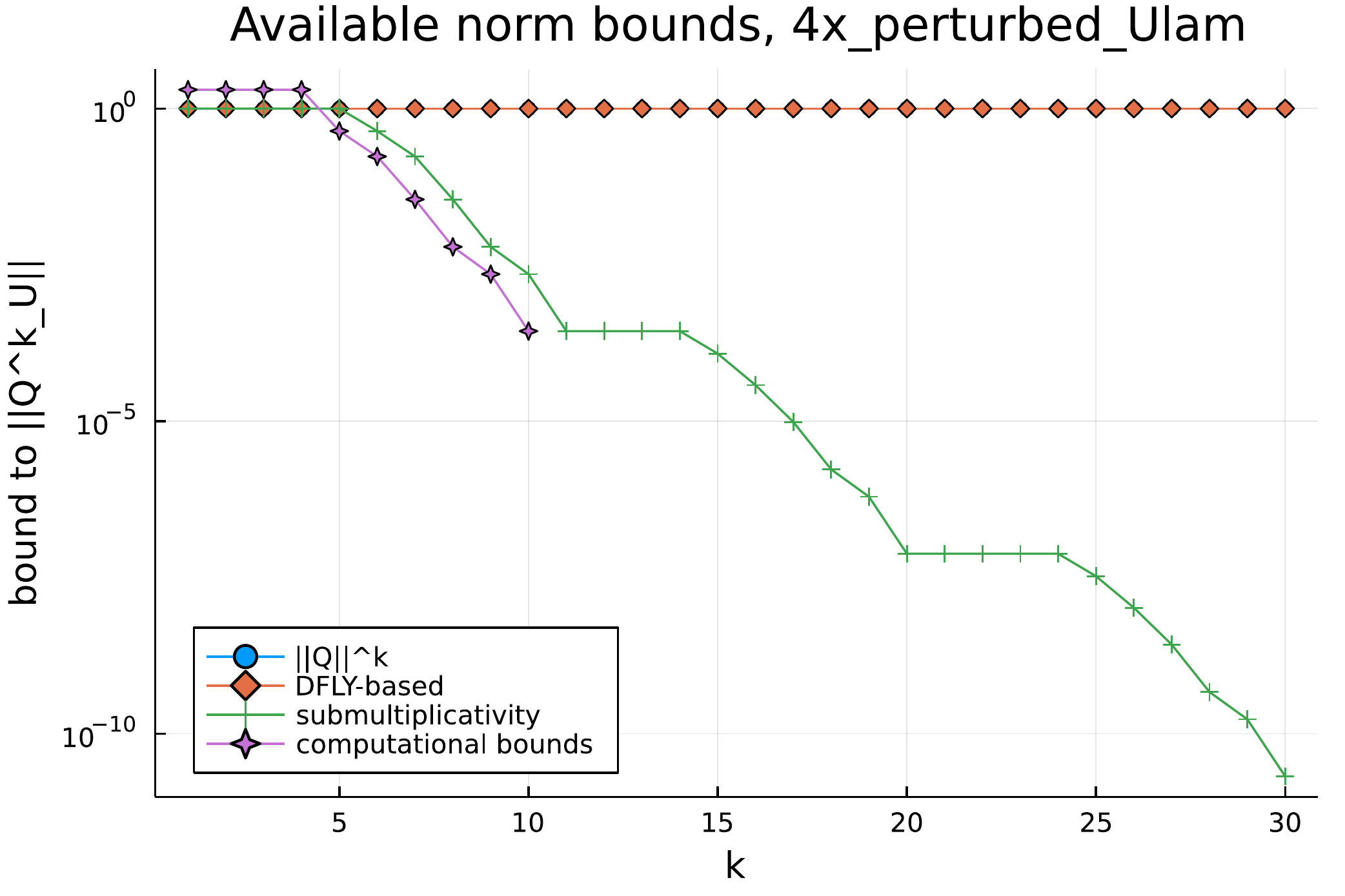}
\includegraphics[width=0.48\textwidth]{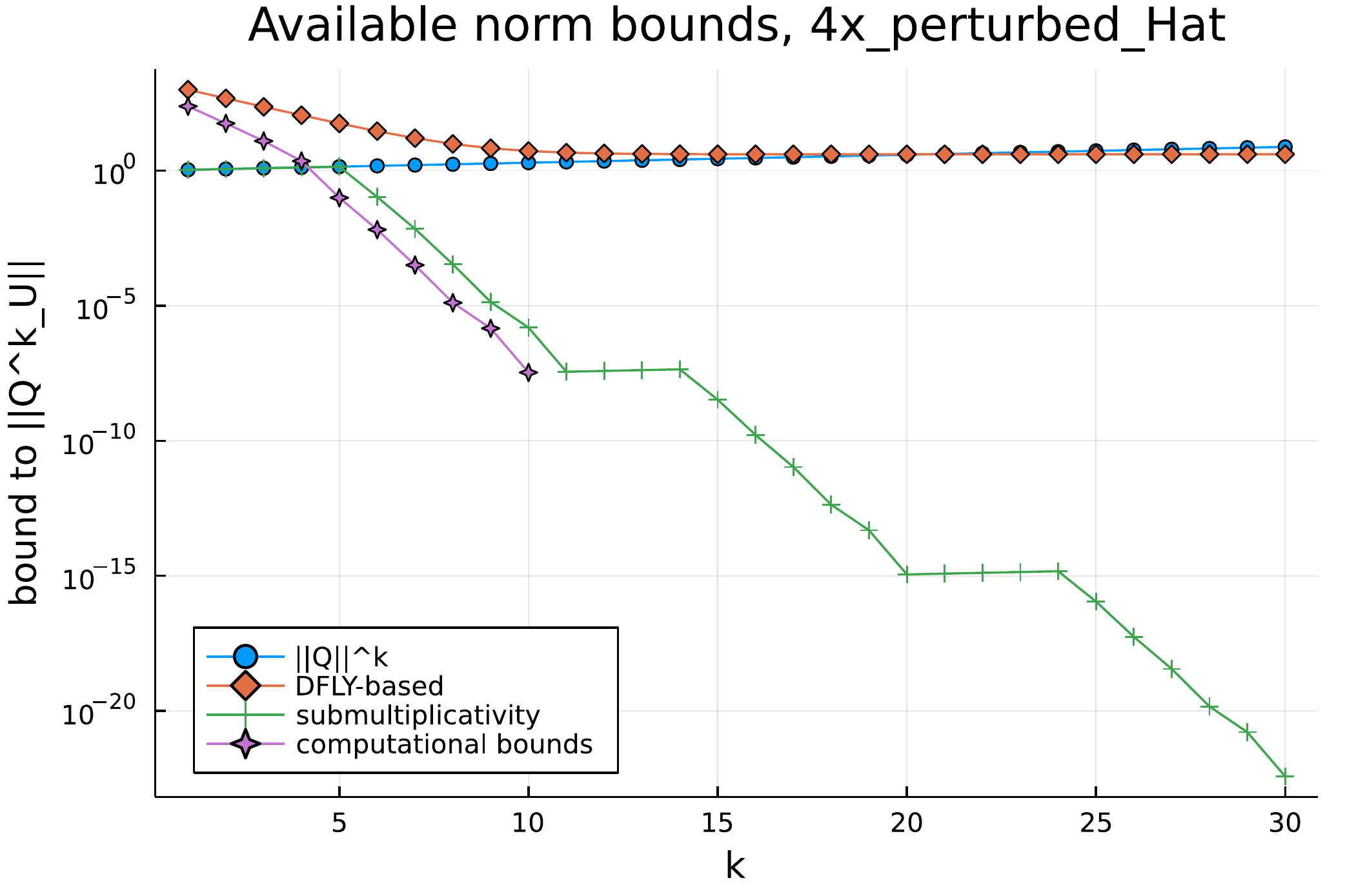}
\caption{Comparison of the norm bounds obtained from various sources for the Ulam (left) and piecewise linear (right) discretization of $T(x) = 4x + 0.01 \sin(8\pi x) \mod 1$ with $n=1024$; they are obtained with $k_{\max}=10$ computational norm bounds and do not rely on coarser grids.} \label{fig:normcompare1}
\end{figure}
\begin{figure}
\centering
\includegraphics[width=0.48\textwidth]{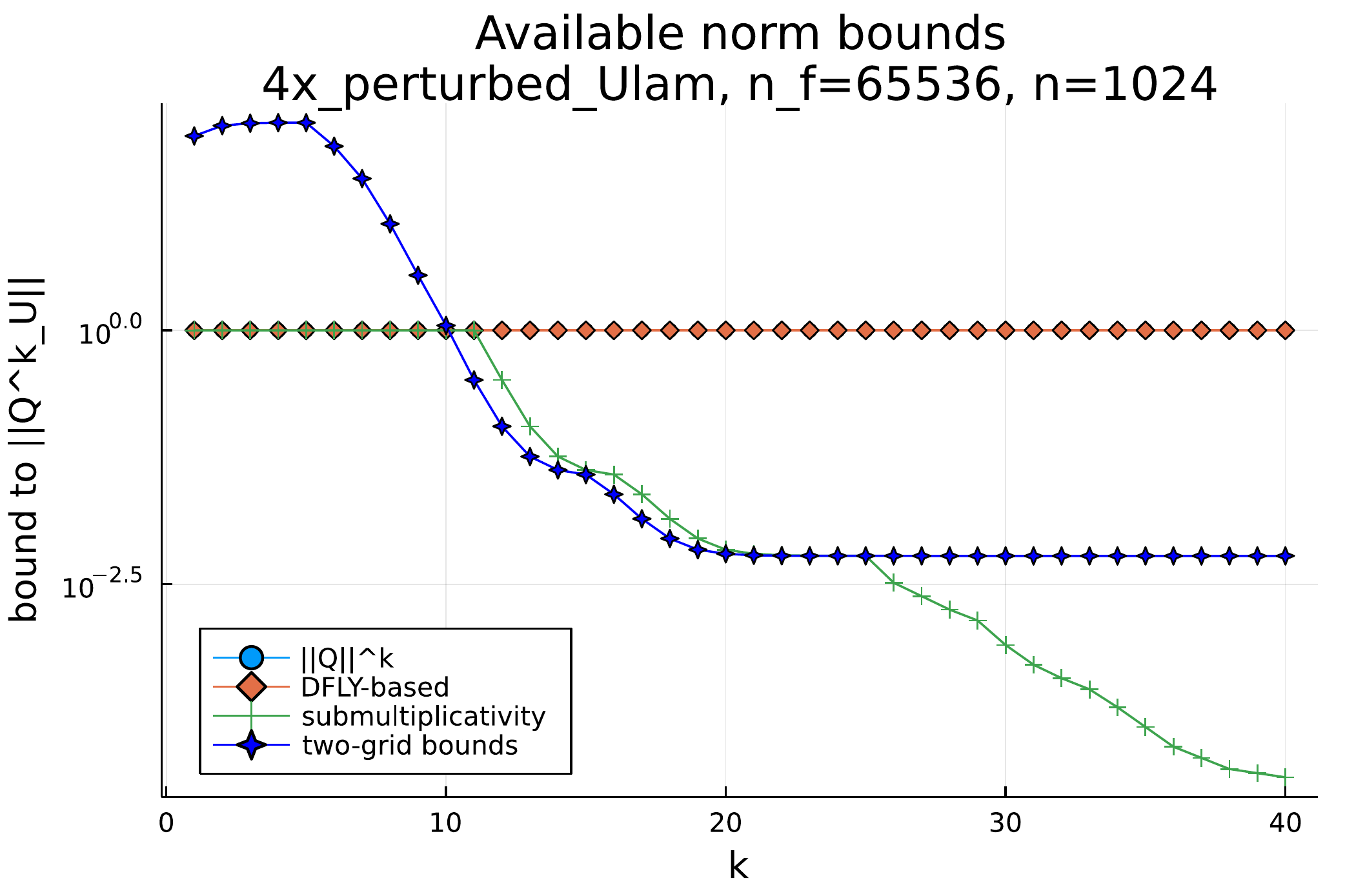}
\includegraphics[width=0.48\textwidth]{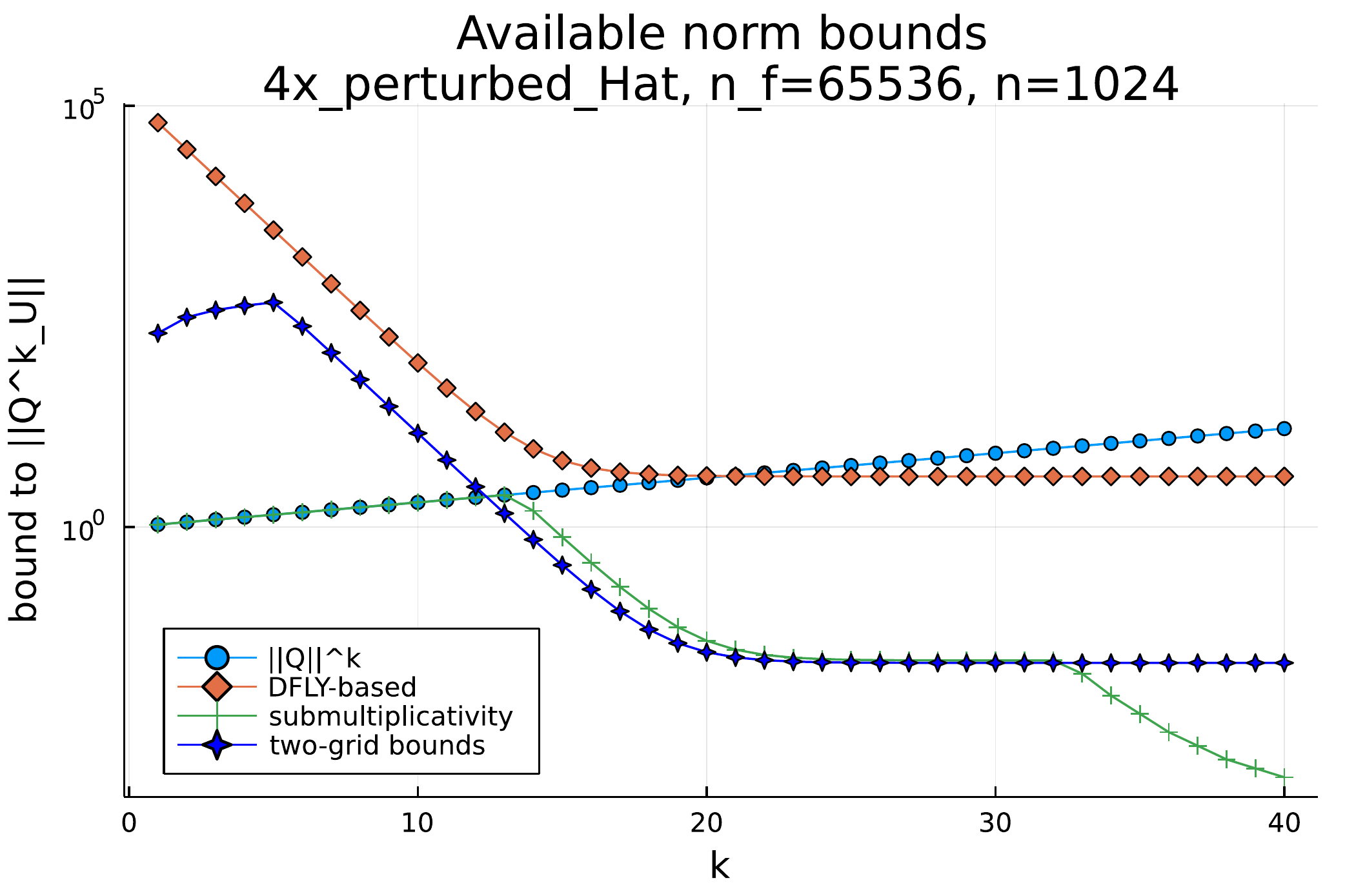}
\caption{Comparison of the norm bounds obtained from various sources for the Ulam (left) and piecewise linear (right) discretization of $T(x) = 4x + 0.01 \sin(8\pi x) \mod 1$ with $n_F=65536$; they are obtained from $k_{\max}=10$ computational norm bounds on the coarse grid with $n=1024$.} \label{fig:normcompare2}
\end{figure}

We note that the two-grid strategy is not guaranteed to succeed and yield a bound $C_m^F < 1$ for some $m$: in particular, when $n$ is too small (and $h$ too large), the second term in the right-hand side of~\eqref{finebounds} is greater than $1$ even for large values of $m$. An example is shown in Figure~\ref{fig:normcompare_lorenz}.
\begin{figure}
\centering
\includegraphics[width=0.48\textwidth]{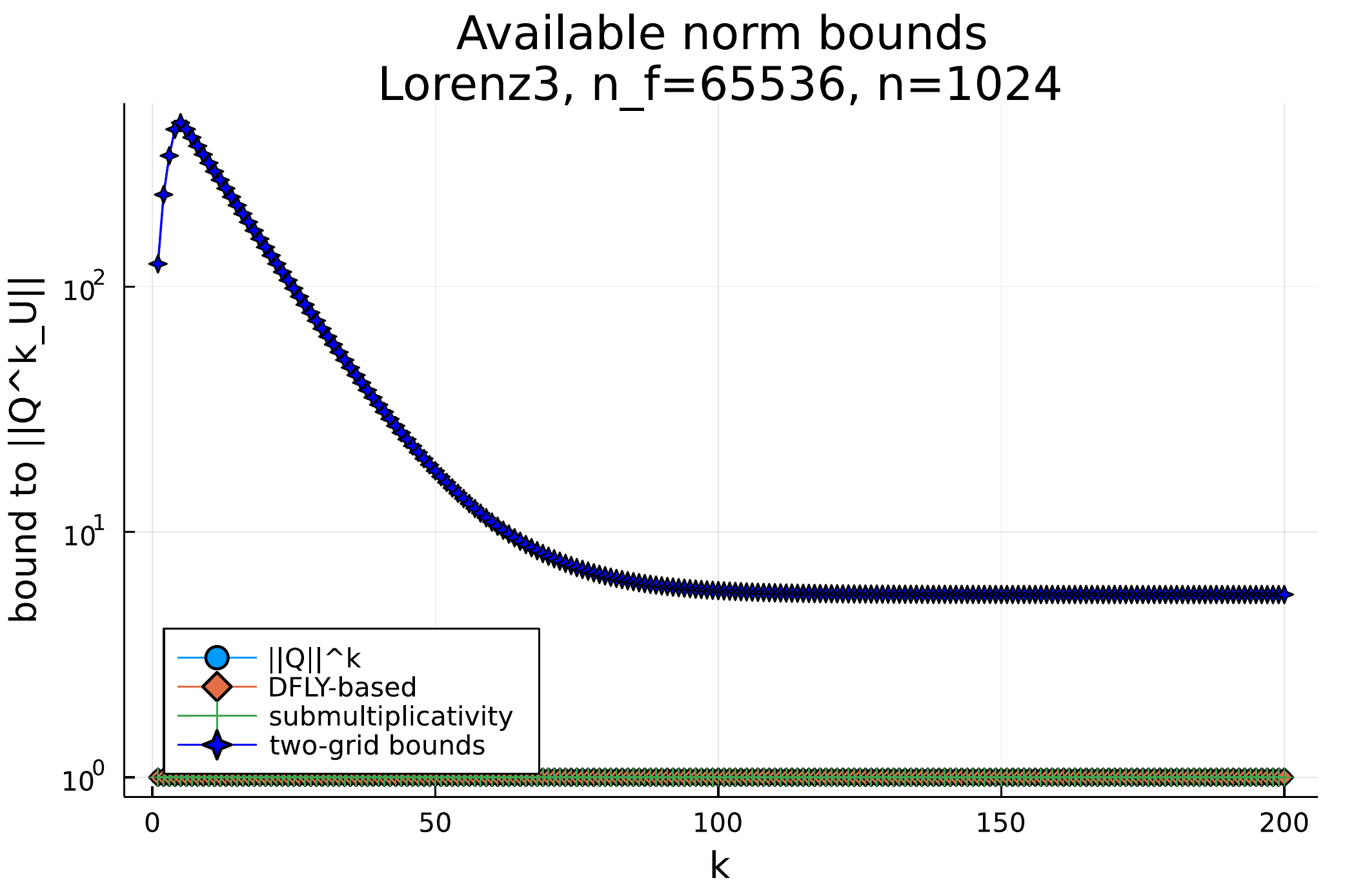}
\caption{Norm bounds obtained from various sources for the Ulam discretization of the third iterate of the Lorenz map~\eqref{lorenzmap} with $n_F=65536$; they are obtained from $k_{\max}=10$ computational norm bounds on the coarse grid with $n=1024$.} \label{fig:normcompare_lorenz}
\end{figure}

\subsection{The algorithms} \label{sec:steps}
 
Putting everything together, we can formulate the following algorithms. To compute a \emph{one-grid} bound for a dynamic using a discretization with $n$ equal intervals, we
\begin{enumerate}
  \item (DFLY coefficients) Compute the coefficients $A,B$ of the Lasota-Yorke inequality~\eqref{genericLY}. This computation requires finding rigorous bounds on the $T'$ and the distorsion $T''/(T')^2$ on each branch of the dynamic, via interval optimization. Its cost does not depend on the discretization size $n$.
  \item (matrix assembly) Construct an interval sparse matrix $\mathbf{L} \ni L_h$ with Algorithm~\ref{algo:assembly}. Its cost is $O(bn)$. \label{step:assembly}
  \item (eigenvalue computation) Compute an approximated eigenvector $Q_h \tilde{u}_h \approx \tilde{u}_h$ using the restarted Arnoldi method in machine arithmetic. Also compute rigorous bounds $\varepsilon_1 \geq \norm{Q_h\tilde{u}_h-\tilde{u}_h}$ and $\varepsilon_2 \geq \abs{i^*\tilde{u}_h-1}$ which will be needed in~\ref{cor:fixedpoint2}.
  \item (norms of powers) Compute norm bounds $\norm{Q_h^{k}|_{\mathcal{U}^0_h}} \leq C_k $ for $k=1,2,\dots,k_{\max}$, using Algorithm~\ref{algo:powernorms} to obtain some first computational bounds and the techniques in Section~\ref{sec:aggregating} to refine them. The value of $k_{\max}$ chosen must be sufficient to obtain $C_{k_{\max}} < 1$; if this inequality does not hold, we can repeat the computation with a larger value of $k_{\max}$. If we choose to multiply by $2$ the value of $k_{\max}$ at each restart, then the cost of this step is $O(n^2z k_{\max})$, with $z \sim b$ and $k_{\max} \sim \log n$ (by the arguments in Section~\ref{sec:works}). \label{step:norms}
  Assuming a constant number of iterations suffices, its cost is $O(nz)$.
  \item (error estimation) Using interval arithmetic or directed rounding to get rigorous bounds, compute the bound for $\norm{u-\tilde{u}_h}$ in~\ref{cor:fixedpoint2}. The cost for this step is merely $O(k_{\max})$, since $\varepsilon_1$ and $\varepsilon_2$ have already been computed.
\end{enumerate}
The computational cost of this algorithm scales as $O(n^2 \log n)$, seriously limiting its usefulness when large values of $n$ are required. To reduce the cost, we can compute instead a \emph{two-grid} bound as follows, using a coarse grid with $n_C$ equal intervals and a fine grid with $n_F$ equal intervals.
\begin{enumerate}
  \item (DFLY coefficients) Compute the coefficients $A,B$ of the Lasota-Yorke inequality~\eqref{genericLY}, as above.
  \item (coarse matrix+norms) Perform steps~\ref{step:assembly} and~\ref{step:norms} of the previous algorithm with $n = n_C$. The cost is $O(n_C^2 b \log n_C)$ as argued above.
  \item (matrix assembly) Construct $\mathbf{L} \ni L_{h_F}$ with Algorithm~\ref{algo:assembly}. Its cost is $O(bn_F)$.
  \item (eigenvalue computation) Compute an approximated eigenvector $Q_{h_F} \tilde{u}_{h_F} \approx \tilde{u}_{h_F}$, as well as $\varepsilon_1$ and $\varepsilon_2$ as above. This step costs $O(n_F z)$.
  \item (error estimation) Compute norm bounds $C_{k,F}$ using the techniques in Section~\ref{sec:aggregating}, with Step~4b instead of~4a, and use them to compute a bound for $\norm{u-\tilde{u}_{h_F}}$ using~\ref{cor:fixedpoint2} (with $n=n_F$). This step costs $O(k_{\max})$.
\end{enumerate}
The total cost depends quadratically on $n_C$, but only linearly on $n_F$. We shall see that this algorithm outperforms the one-grid strategy for suitable values of~$n_C$ and~$n_F$.

\section{Numerical experiments}\label{sec:9}

The proposed algorithm has been implemented in the Julia language for both the Ulam (Section~\ref{sec:Ulam}) and piecewise linear projection (Section~\ref{sec:PL}). Our code is available on~\url{https://github.com/JuliaDynamics/RigorousInvariantMeasures.jl}. The following numerical experiments have been performed with Julia 1.7.1 on an Imac i7-4790K \@ 4.00GHz.

\subsection{The Lanford map}
As a first experiment, we compute the invariant measure of 
\begin{equation} \label{lanford}
  T: [0,1] \to [0,1], \quad T(x) = 2x + \frac12x(1-x) \mod 1  
\end{equation}
with the Ulam projection. We tested both the one-grid described above, with various powers of 2 as the values of $n$, and the two-grid bound, with $n_C=1024$ and various powers of 2 as the values of $n_F$. We display in Figure~\ref{fig:Lanford-breakdown} the rigorous error bounds on $\norm{u - \tilde{u}_h}_{L^1}$ that have been proved, and a breakdown of how the CPU time is divided between the steps of each algorithm described in Section~\ref{sec:steps}.
\begin{figure}
\centering
\includegraphics[width=\textwidth]{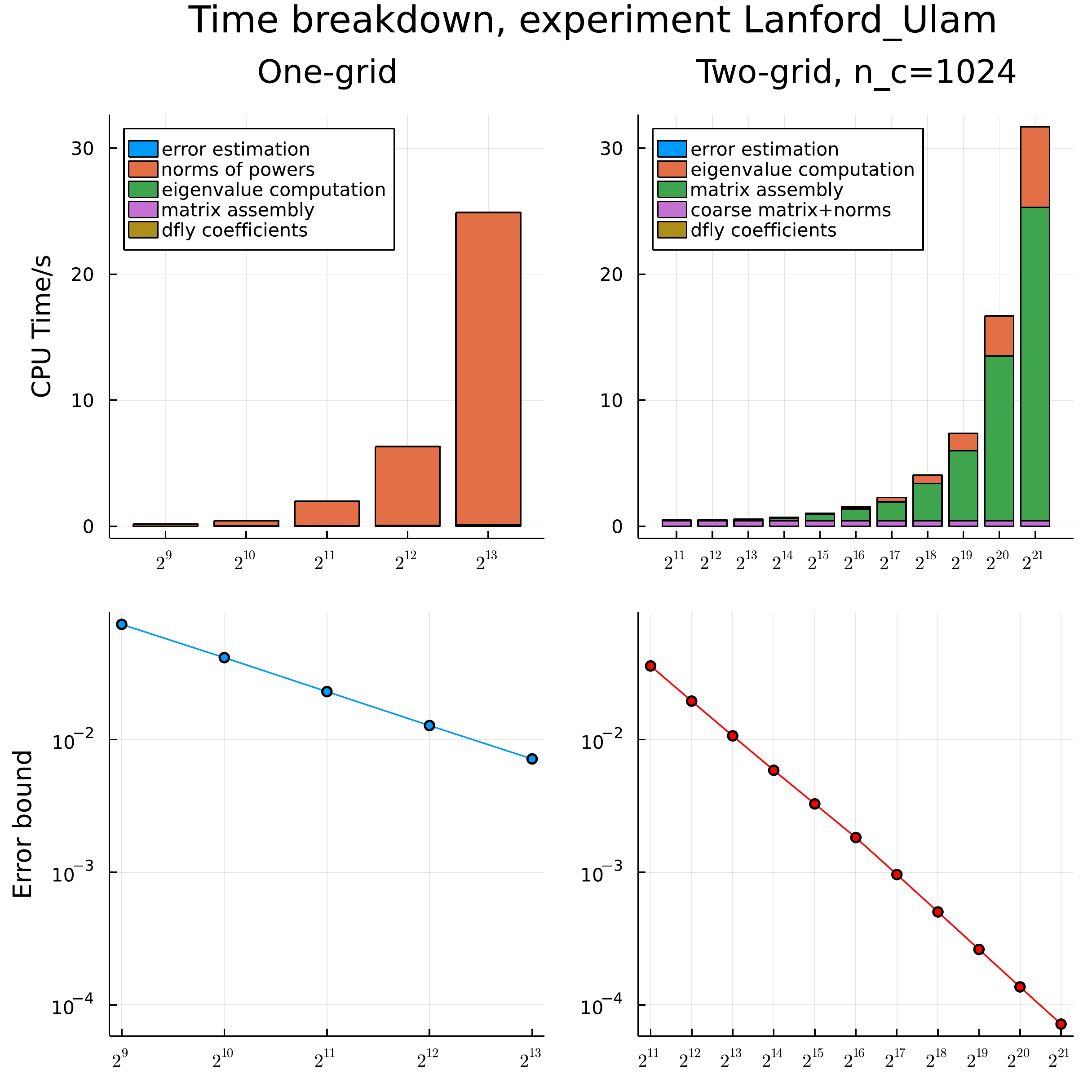}
\caption{Error bounds proved and CPU time breakdown along the five steps of the algorithms in Section~\ref{sec:steps} for the Lanford map~\eqref{lanford}.} \label{fig:Lanford-breakdown}
\end{figure}
Bounds on the same quantity have been computed in~\cite{GaNi}, but working on the iterate $T^2$ in place of $T$ was necessary there, because the inequality~\cite[Theorem~5.2]{GaNi} there is weaker than Theorem~\ref{thm:VarLY} here. The major innovation in this work is the two-grid strategy, which allows to prove bounds as small as $10^{-4}$ in less than one minute of CPU time. With the two-grid strategy (on the right), larger dimensions can be used, and the majority of time is spent assembling the matrix $\mathbf{L}_{h_F}$ and computing its fixed point vector.

A detailed analysis of the tradeoff between error bound and CPU time obtained with various choices of $n, n_C, n_F$ is shown in Figure~\ref{fig:Lanford-time}. \begin{figure}
\centering
\includegraphics[width=\textwidth]{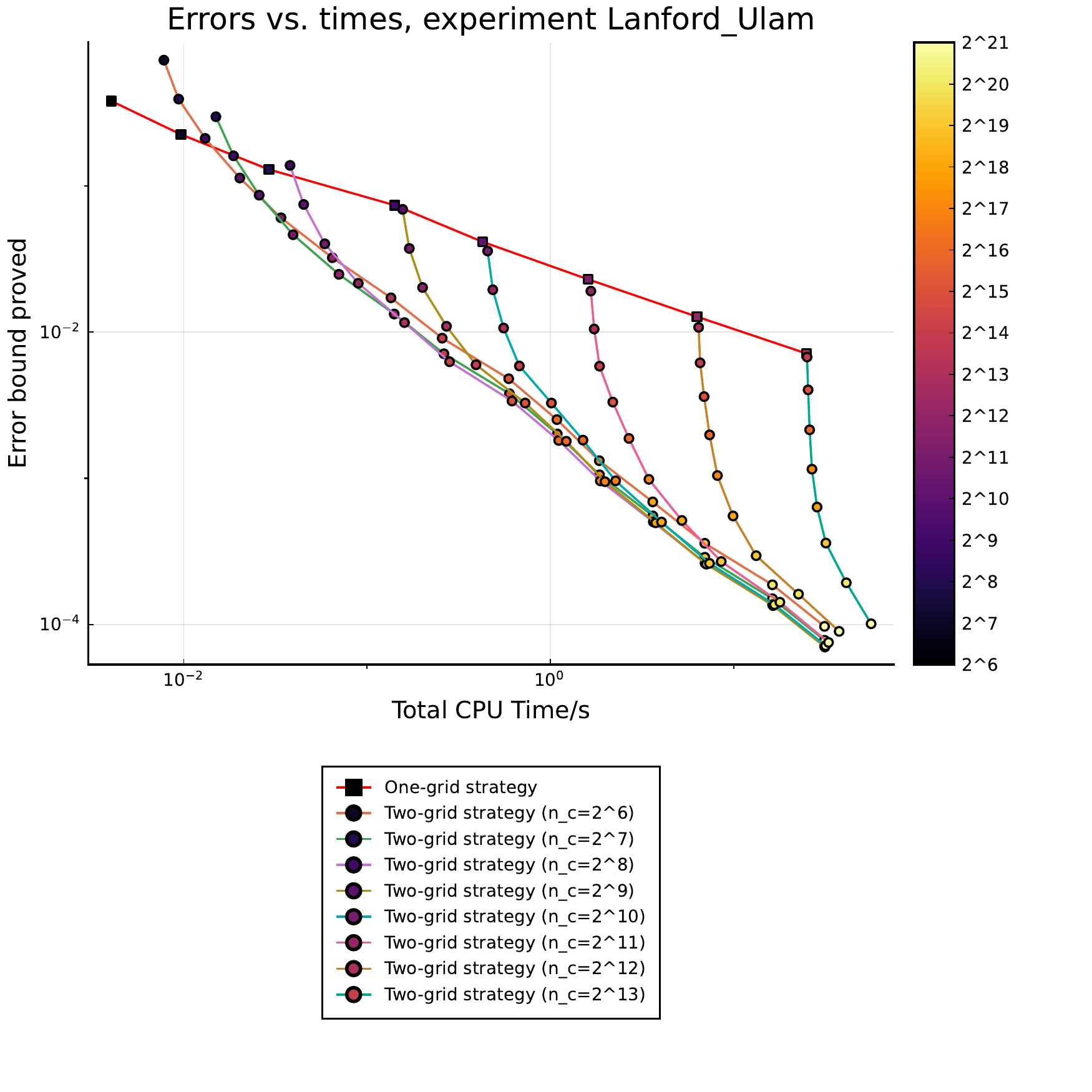}
\caption{Error bound vs.~time for various choices or $n$ and $n_F$, for the Lanford map~\eqref{lanford}. The marker color represents the value of $n$ or $n_F$.} \label{fig:Lanford-time}
\end{figure}
One can see from this plot that the error scales approximately as $t^{-1/2}$ with the one-grid strategy, and approximately as $t^{-1}$ with the two-grid strategy, as predicted by our complexity estimates. After an initial period to amortize the power norm computation, all sufficiently large choices of $n$ have similar asymptotic efficiency; this suggests that to improve the precision of an estimate it is better to keep $n$ constant and increase the value of $n_F$.

Using the tecnique above we were able to compute an enclosure for the Lyapunov exponent of the Lanford map
\[
\int \log(|T'|)\,df \in [0.657657, 0.657667]
\]
where the diameter of enclosure is $9.45\cdot 10^{-6}$.
This estimate was produced with $n_C=2^{11}$ and $n_F = 2^{25}$ in $1476$ seconds; most of this time was spent assembling the matrix $\mathbf{L}_{h_F}$.

\subsection{A non-linear non-Markov map}
We consider the following nonlinear modification of $\frac{17}{5}x \mod 1$:
\begin{equation} \label{nonlinearnonmarkov}
  T(x) = \begin{cases}
\frac{17}{5}x & 0\leq x \leq \frac{17}{5},\\
\frac{34}{25}(x-\frac{5}{17})^2 + 3(x-\frac{5}{17}), & \frac{5}{17} < x \leq \frac{10}{17},\\
\frac{34}{25}(x-\frac{10}{17})^2 + 3(x-\frac{10}{17}), & \frac{10}{17} < x \leq \frac{15}{17},\\
\frac{17}{5}(x-\frac{15}{17}) & \frac{15}{17} < x \leq 1,
\end{cases}  
\end{equation}
again with the Ulam projection. This is another of the dynamics considered in~\cite{GaNi}, this time without modification. We display the same information in Figures~\ref{fig:175_nonlinear-breakdown} and~\ref{fig:175_nonlinear-time}.
\begin{figure}
\centering
\includegraphics[width=\textwidth]{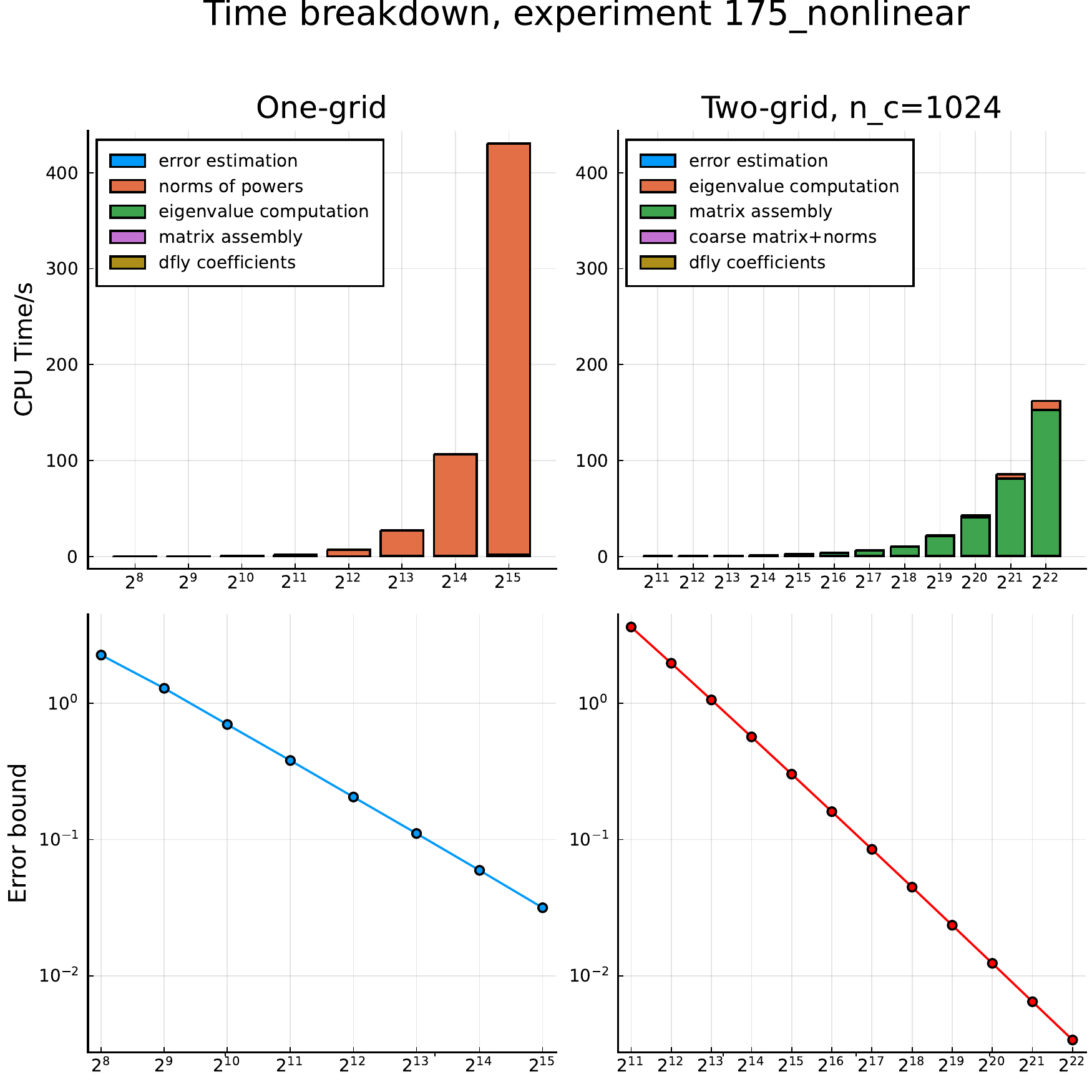}
\caption{Error bounds proved and CPU time breakdown for the non-Markov map~\eqref{nonlinearnonmarkov}.} \label{fig:175_nonlinear-breakdown}
\end{figure}
\begin{figure}
\centering
\includegraphics[width=\textwidth]{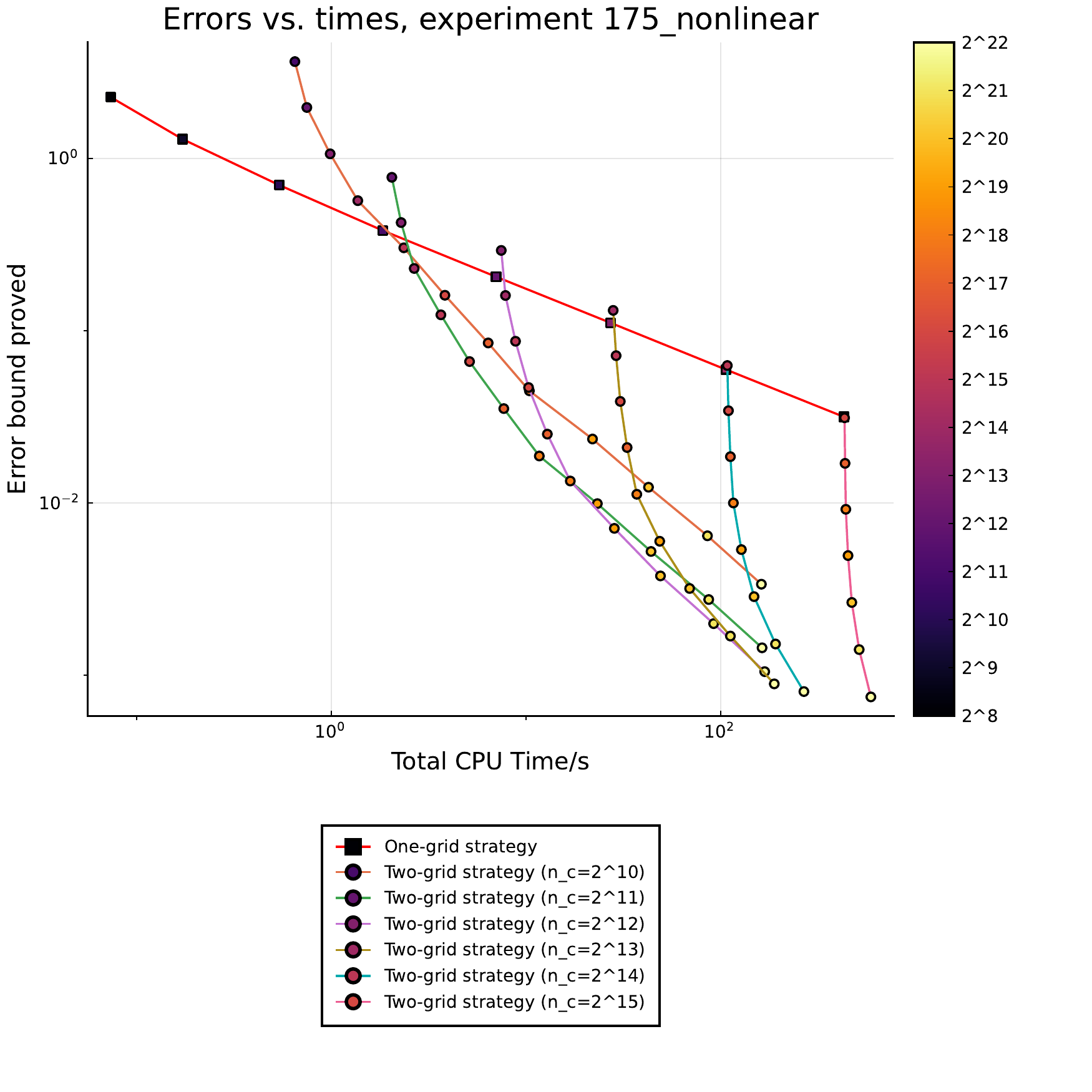}
\caption{Error bound vs.~time for various choices of $n$ and $n_F$ on the non-Markov map~\eqref{nonlinearnonmarkov}.} \label{fig:175_nonlinear-time}
\end{figure}
This experiment is more challenging, especially since $n_C = 2^{10}$ is required to reach a bound $C_{k_{\max}} < 1$ with the two-grid strategy, but the same features appear in the plots, highlighting in particular the massive improvements provided by the two-grid strategy.

Using the tecnique above we were able to compute an enclosure for the Lyapunov exponent of this
map
\[
\int \log(|T'|)\,df \in [1.21933, 1.22016]
\]
where the diameter of enclosure is $0.00082$. The computation time to obtain such an approximation was $2110$ seconds.

\subsection{A Markov perturbation of $4x \mod 1$}

The next example we consider is 
\begin{equation}
  T(x) = 4x + 0.01 \sin(8\pi x) \mod 1. \label{4x_perturbed}  
\end{equation}
In this experiment, we use the piecewise linear discretization to provide a bound to $\norm{u - \tilde{u}_h}_{L^\infty}$ in the $L^\infty$ norm, again replicating an example in~\cite{GaNi}. The results are reported in Figures~\ref{fig:4x_perturbed_Hat-breakdown} and~\ref{fig:4x_perturbed_Hat-time}.
\begin{figure}
\centering
\includegraphics[width=\textwidth]{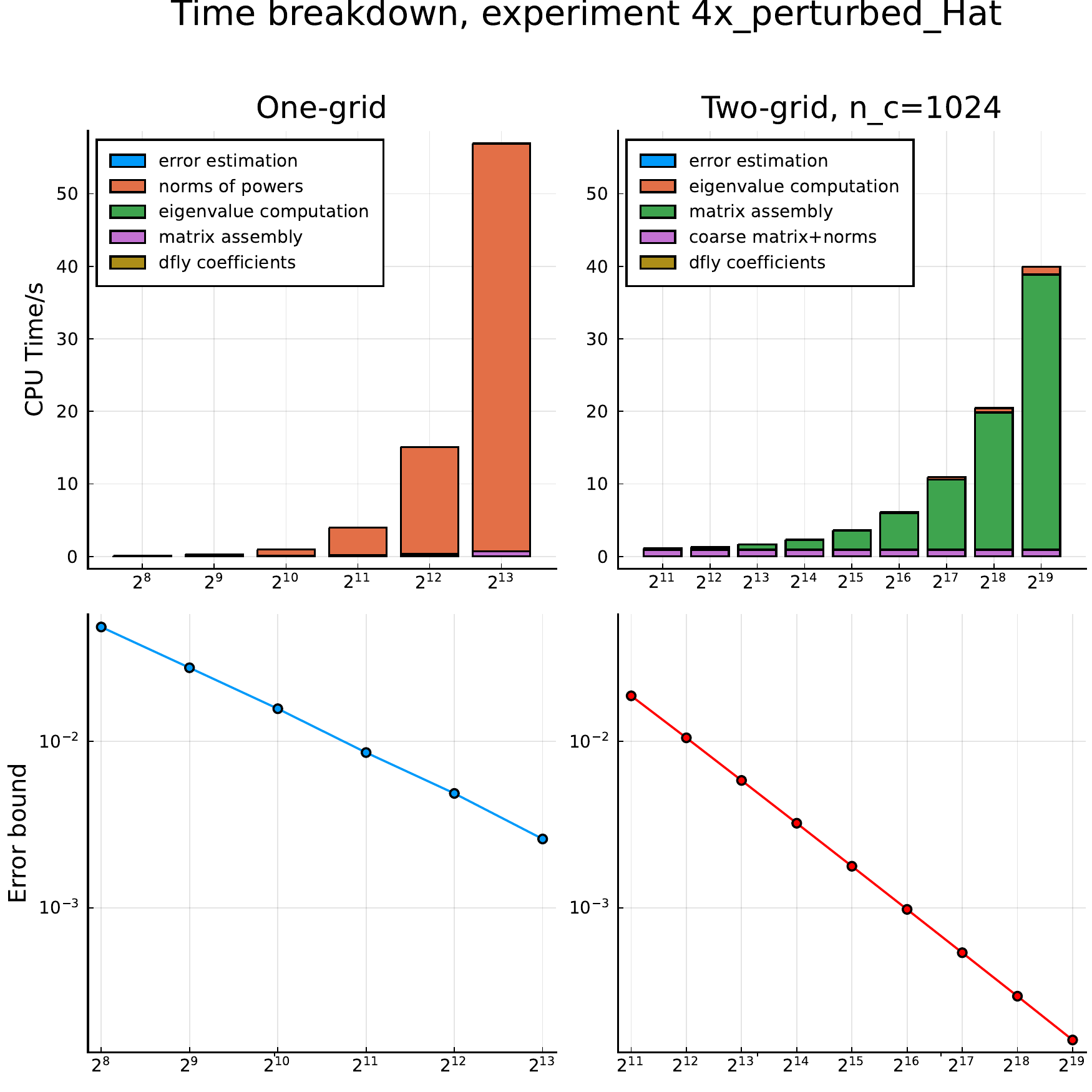}
\caption{Error bounds proved and CPU time breakdown for the ``$4x$ perturbed'' map~\eqref{4x_perturbed}.} \label{fig:4x_perturbed_Hat-breakdown}
\end{figure}
\begin{figure}
\centering
\includegraphics[width=\textwidth]{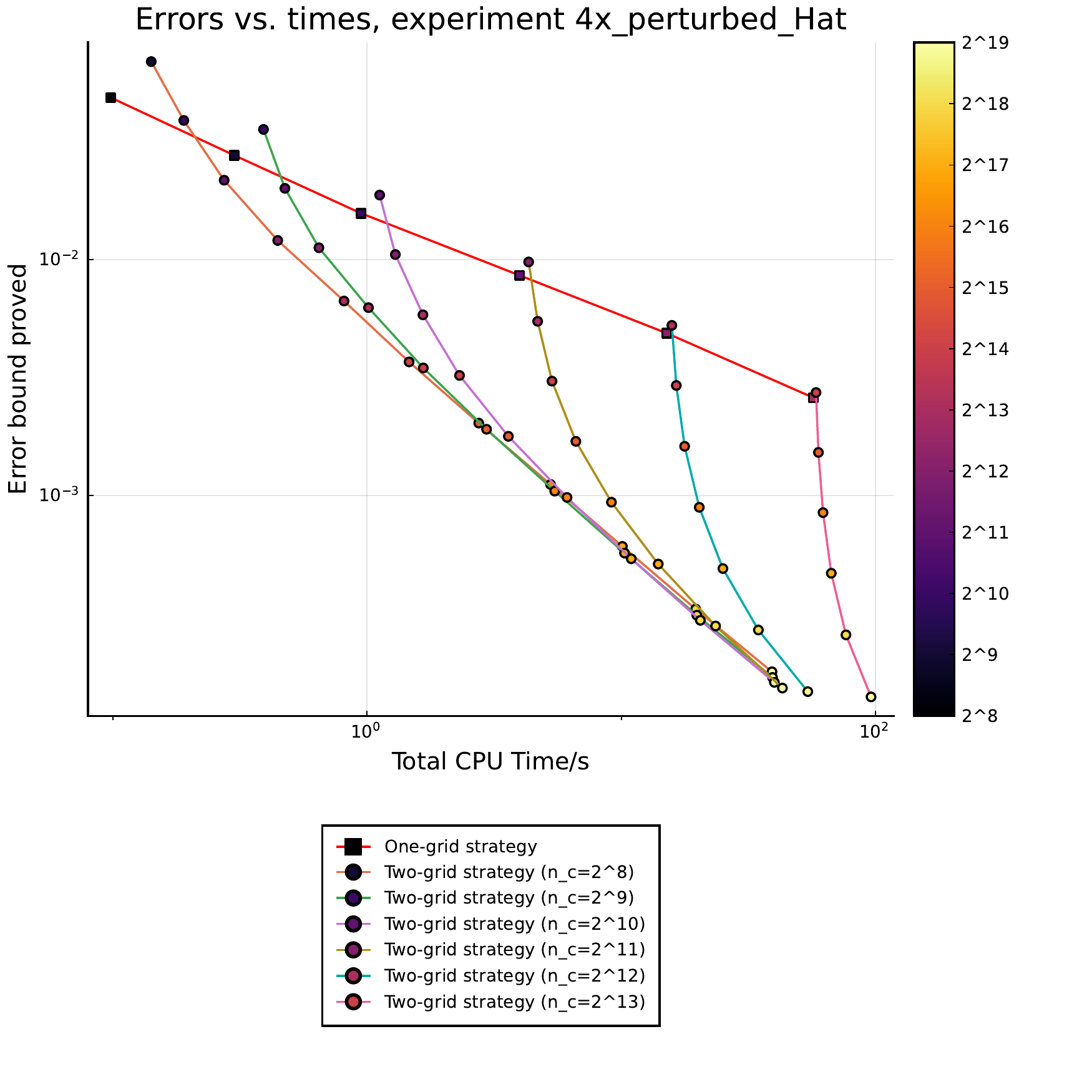}
\caption{Error bound vs.~time for various choices of $n$ and $n_F$ on the ``$4x$ perturbed'' map~\eqref{4x_perturbed}.} \label{fig:4x_perturbed_Hat-time}
\end{figure}

Despite the different projection, the workload and results are very similar. Note that assembling the matrix $\mathbf{L}$ is more expensive than in the other examples; this is not related to the different projection, but it is due to the fact that providing certified enclosures for trigonometric functions is computationally expensive.

Using the tecnique above we were able to compute an enclosure for the Lyapunov exponent of this map
\[
\int \log(|T'|)\,df \in [1.38530, 1.38531]    
\]
where the diameter of enclosure is $3.3\cdot 10^{-6}$.
The computation time to obtain such an approximation was $3016$ seconds.

\subsection{One-dimensional Lorenz map}

The Lorenz system is a famous example of a $3$-dimensional vector flows that, presents a strange attractor.
We refer to \cite{AP} for a historical introduction 
to the geometric model of the Lorenz system 
and a careful presentation of its construction; 
the example we present in this subsection is the one-dimensional map associated to the stable 
foliation of the geometric Lorenz system studied in \cite{GaNi2}.

This map is 
\begin{equation} \label{lorenzmap}
T(x) = \begin{cases}
\theta \left|x-\frac{1}{2}\right|^{\alpha} & 0\leq x < \frac{1}{2},\\
 1- \theta\left|x-\frac{1}{2}\right|^{\alpha} & \frac{1}{2}< x \leq 1,\\
\end{cases}
\end{equation}
with $\alpha=51/64$ and $\theta=109/64$.
Note that the derivative of this map goes to $\infty$ as we approach $1/2$, so the one-step Lasota-Yorke inequality which we have been using in the other examples does not hold; by direct computation, one sees that $T''/(T')^2$ behaves as 
$|x-1/2|^{-\alpha}$ near $1/2$, and hence it is unbounded.

\begin{lemma}
Let $T:[0,1]\to [0,1]$ and suppose there exists a finite partition
$\{P_k\}_{k=1}^b$ of $[0,1]$ such that
\begin{enumerate}
\item $T_k = T|_{P_k}$ is $C^2$,
\item $|T'(x)|>2$ for all $x\in [0,1]$.
\end{enumerate}     
Let $I_{l}=\{x\mid |T''/(T')^2|\geq l\}$ and suppose there exists an $l$ such that
\[
A = \frac{1}{2}\int_{I_l}\left|\frac{T''}{(T')^2}\right|dm+\frac{2}{\inf(|T'|)}<1,    
\]
then
\[
\Var{Lf}\leq A \Var(f)+\left(\max_k\frac{2}{|P_k|}+l\right)||f||_{L^1}. 
\]
\end{lemma} 

To prove a one-step Lasota-Yorke inequality for our example, we applied this lemma to the third iterate of the map $T$.
The coefficients in the obtained inequality are large ($A \approx 0.922, B\approx 48.43$) and quite expensive to compute (about one minute). 

The results obtained are presented in Figures~\ref{fig:lorenz-breakdown} and~\eqref{fig:lorenz-time}.
\begin{figure}
\centering
\includegraphics[width=\textwidth]{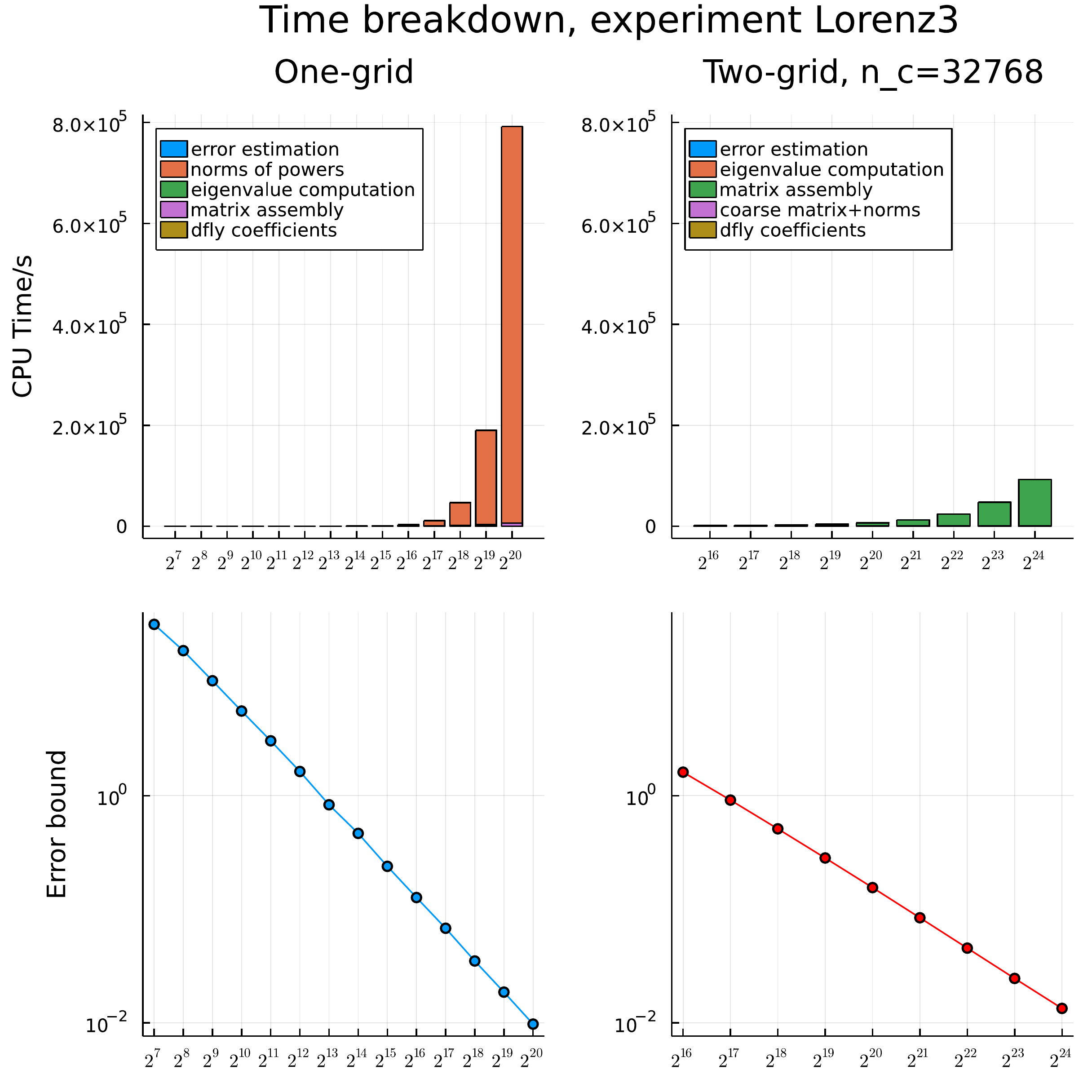}
\caption{Error bounds proved and CPU time breakdown for the third iterate of the Lorenz map~\eqref{lorenzmap}.} \label{fig:lorenz-breakdown}
\end{figure}
\begin{figure}
\centering
\includegraphics[width=\textwidth]{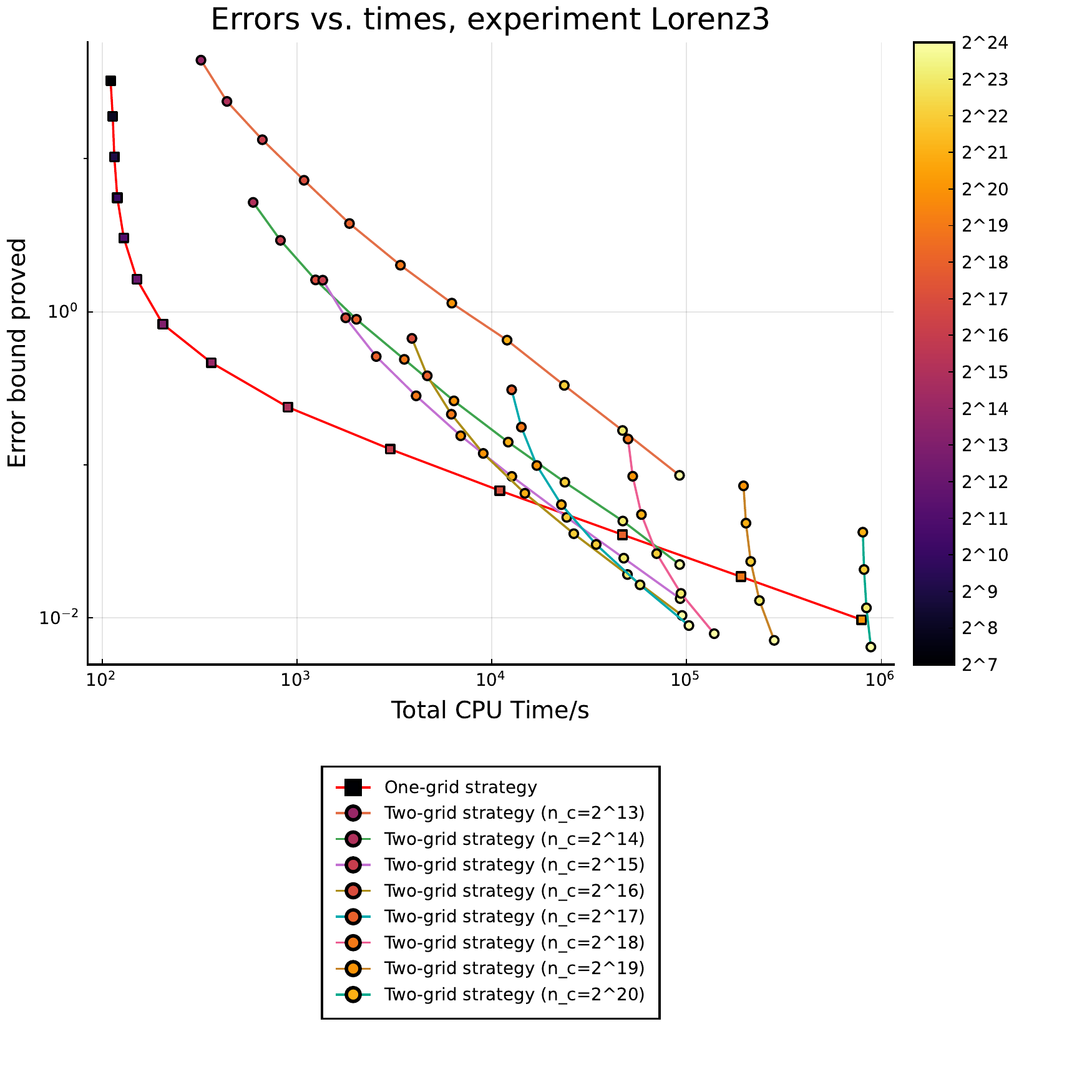}
\caption{Error bound vs.~time for various choices of $n$ and $n_F$ on the third iterate of the Lorenz map~\eqref{lorenzmap}.} \label{fig:lorenz-time}
\end{figure}
One can see that the error bounds are generally worse than those obtained with the previous maps; in particular, we need to use a larger value of the coarse discretization size $n$. Indeed, the two-grid strategy fails to produce useful bounds when used with $n=1024$: due to the large value of $B$, with this choice of $h$ the formula~\eqref{finebounds} produces only bounds for $Q_{h_F}^m|\mathcal{U}^0_{h_F}$ that are larger than 1, hence the convergence of the series appearing in~\eqref{theorem1_estimate} cannot be proved with Lemma~\ref{lem:onembound} and the method fails. Nevertheless, larger values of $n$ and $n_F$ yields valid bounds for the error, as shown in Figure~\ref{fig:lorenz-time}; the two-grid strategy eventually surpasses the efficiency of the one-grid bounds, and for instance it is faster by an order of magnitude when one seeks to prove an error bound of $10^{-2}$.

Using the tecnique above with $n_C=2^{17}$, $n_F = 2^{24}$, we can compute an enclosure for the Lyapunov exponent of this map
\[
\int \log(|T'|)\,df \in [0.580676, 0.786467]
\]
where the diameter of enclosure is $0.2058$. The computation time to obtain such an approximation was 101828 seconds.

\subsection{Limitations of machine arithmetic}
In several computations involved in our algorithm, 
floating point arithmetic gives a lower bound on the attainable precision:
\begin{itemize}
    \item the diameter of the interval entries of the interval matrix representing the discretized operator 
    is generically bounded below by machine precision,
    \item machine floating point arithmetic is going to be the main source of the error stemming from the computation of the residual
    $||Pu_h-u_h||$.
\end{itemize} 
A possible strategy to overcome machine arithmetic limitations could be to first compute 
a coarse approximation in machine arithmetic, allowing us to estimate mixing rates $C_k$,
and then compute a finer approximation in higher precision floating point arithmetic, i.e.,
a ``low-precision coarse  -- high-precision fine '' scheme.

While this corresponds to a small modification of the code, no experiments have been done in this direction.

Another, much more serious problem arising from machine precision is 
the numerical error arising in our norm estimates. 
If the discretized operator is not sparse, it may be impossible 
to prove that one of its iterates contracts $\mathcal{U}_0$,
due to the estimates we need to put in place to guarantee an upper bound 
of the norm, see subsection \ref{subsec:boundmachine}.
This can also be solved by using higher precision floating point numbers,
but the computational overhead would be difficult to manage.

\section{Final remarks and considerations}\label{sec:10}
In this paper we introduced a general framework for the approximation of invariant measures. We gave a finite set of inequalities that, once proved, give rise to an algorithm for the approximation, once we can prove computationally the existence of an $m$ such
that $\norm{Q_h^m|_{\mathcal{U}_0}} \leq C_m<1$. 
% We obtained rigorous bounds on the quantities $C_k$ by implementing Algorithm~\ref{algo:powernorms} and keeping track of the numerical errors introduced according to the previous lemmas; note that one needs to account for all the numerical errors in the sums appearing in rows~\ref{row:error2} and~\ref{row:error4}, even those hidden inside $\norm{v}_{\ell^1}$.

On the computational side, the major contribution of this paper is the new ``coarse-fine'' framework based on two discretizations with grids of different sizes; this framework greatly reduces the computational burden of the estimation algorithms introduced in~\cite{GaNi}. The experiments in~\cite{GaNi} relied on computational norm estimation with Algorithm~\ref{algo:powernorms}, which requires $O(mn^2)$ floating point operations to obtain estimates $C_k$ for $k\leq m$. Typically, $n\approx 10^5$ to $10^6$ is needed to get a meaningful estimate, so this computation was doable, but extremely slow. Here, we give a strategy to combine bounds from various sources in Section~\ref{sec:aggregating}, including in particular those coming from the coarse-fine strategy~\eqref{finebounds}. This improvement gives a major reduction in the computational time: while the results in~\cite{GaNi} were obtained on a supercomputing cluster, we can replicate them in a few minutes on a common laptop computer.

\bibliographystyle{plain}
\bibliography{references}

% \begin{thebibliography}{99}
% 
% \end{thebibliography}

\end{document}